\documentclass[a4paper,10pt]{article}
\usepackage{amsmath,amssymb,amsthm}
\usepackage{mathrsfs}
\usepackage{color}
\usepackage{ascmac}
\usepackage{dsfont}
\usepackage{xcolor}
\theoremstyle{definition}
 \newtheorem{dfn}{Definition}[section]
 \newtheorem{remark}[dfn]{Remark}  
\theoremstyle{plain}
 \newtheorem{thm}[dfn]{Theorem}
 
 \newtheorem{lem}[dfn]{Lemma}
 \newtheorem{cor}[dfn]{Corollary}

\numberwithin{equation}{section}
\newcommand{\sP}{\mathscr{P}}

\newcommand{\sS}{\mathscr{S}}
\newcommand{\sR}{\mathscr{R}}
\newcommand{\sL}{\mathscr{L}}
\newcommand{\sN}{\mathscr{N}}

\newcommand{\sF}{\mathscr{F}}
\newcommand{\lr}[1]{\mathrm{L}_{#1}}
\newcommand{\rB}{{\mathrm{B}}}
\newcommand{\rH}{{\mathrm{H}}}
\newcommand{\rL}{{\mathrm{L}}}
\newcommand{\rW}{{\mathrm{W}}}

\newcommand{\dv}{{\rm div}\,}

\newcommand{\BB}{{\mathbb B}}
\newcommand{\BR}{{\mathbb R}}
\newcommand{\BC}{{\mathbb C}} 

\newcommand{\BN}{{\mathbb N}}

\newcommand{\BT}{{\mathbb T}}
\newcommand{\torus}{{\mathbb T}}

\newcommand{\BZ}{{\mathbb Z}}

\newcommand{\CA}{{\mathcal A}}
\newcommand{\CB}{{\mathcal B}}

\newcommand{\CI}{{\mathcal I}}

\newcommand{\CN}{{\mathcal N}}

\newcommand{\CP}{{\mathcal P}}
\newcommand{\CR}{{\mathcal R}}
\newcommand{\CS}{{\mathcal S}}
\newcommand{\CT}{{\mathcal T}}
\newcommand{\CH}{{\mathcal H}}

\newcommand{\CU}{{\mathcal U}}
\newcommand{\CQ}{{\mathcal Q}}
\newcommand{\CV}{{\mathcal V}}
\newcommand{\CW}{{\mathcal W}}
\newcommand{\CX}{{\mathcal X}}
\newcommand{\CY}{{\mathcal Y}}
\newcommand{\CZ}{{\mathcal Z}}
\newcommand{\dd}{{\mathrm d}}
\newcommand{\e}{{\mathrm e}}

\newcommand{\bF}{{\mathbf F}}

\newcommand{\bG}{{\mathbf G}}

\newcommand{\bH}{{\mathbf H}}

\newcommand{\bU}{{\mathbf U}}

\newcommand{\bh}{{\mathbf h}}

\newcommand{\bn}{{\mathbf n}}

\newcommand{\bw}{{\mathbf w}}

\newcommand{\bq}{{\mathbf q}}

\newcommand{\bv}{{\mathbf v}}
\newcommand{\bu}{{\mathbf u}}
\newcommand{\bff}{{\mathbf f}}
\newcommand{\bg}{{\mathbf g}}

\newcommand{\fp}{{\mathfrak p}}
\newcommand{\fq}{{\mathfrak q}}
\newcommand{\fr}{{\mathfrak r}}
\newcommand{\pd}{\partial}

\begin{document}
\title{Periodic $\lr{p}$ estimates by $\sR$-boundedness:\\ Applications to the Navier-Stokes equations}
\author{Thomas Eiter\thanks{Weierstrass Institute for Applied Analysis and 
Stochastics, Mohrenstr. 39, 10117 Berlin, Germany.
\endgraf
e-mail address: thomas.eiter@wias-berlin.de},
\enskip Mads Kyed\thanks{Hochschule Flensburg, Kanzleistra\ss e 91-93, 24943
Flensburg, Germany. \endgraf
e-mail address: mads@kyed.de \endgraf
JA Professor of Waseda University}, and 
Yoshihiro Shibata\thanks{Department of Mathematics,  Waseda University, 
Ohkubo 3-4-1, Shinjuku-ku, Tokyo 169-8555, Japan. \endgraf
e-mail address: yshibata325@gmail.com \endgraf
Adjunct faculty member in the Department of 
Mechanical Engineering and Materials Science,
University of Pittsburgh, USA \endgraf
Partially supported by Top Global University Project, 
JSPS Grant-in-aid for Scientific Research (A) 17H0109, and Toyota Central Research Institute Joint Research Fund }}
\date{}
\maketitle

\begin{abstract}
General evolution equations in Banach spaces are investigated.
Based on an operator-valued version of de Leeuw's transference principle, time-periodic $\lr{p}$ estimates of maximal regularity type
are established from $\sR$-bounds of the family of solution operators
 ($\sR$-solvers) to the corresponding resolvent problems.
With this method, existence of time-periodic solutions to the Navier-Stokes equations is shown for two configurations:
in a periodically moving bounded domain 
and in an exterior domain, subject to prescribed time-periodic forcing
and boundary data.
\end{abstract}

\section{Introduction}\label{sec:1} 
The study of time-periodic solutions to evolution equations is the study of oscillations.
In this article we investigate time-periodic solutions 
corresponding to time-periodic data, that is,
systems of forced oscillation.
A number of different methods, further described below,
are traditionally used to carry out a mathematical investigation of such solutions.
In the following we introduce a new technique to
 establish \textit{a priori} estimates of maximal $L_p$ regularity type for linearized equations.
Such estimates are essential in the study of nonlinear problems, 
which we will demonstrate on some examples.  
For notational simplicity, we consider only $2\pi$-periodic problems.
 By a simple scaling argument, however, all our results
extend to $\CT$-periodic problems for any $\CT>0$.

The study of $2\pi$-time-periodic solutions to evolution
 equations can be carried out in a framework where
the time axis is replaced with a torus $\BT:=\BR/2\pi\BZ$. Consider for example an abstract evolution equation
\begin{equation}\label{abstractevoeq}
\partial_t u + Au=f \ \ \text{in $\BT$}
\end{equation}
in a Banach space $X$, where $A$ is a linear operator on $X$.
Since the time domain is a torus, a solution 
to \eqref{abstractevoeq} is intrinsically time-periodic. 
We refer to estimates of the solution in 
$\lr{p}(\BT; X)$ norms as \emph{periodic $L_p$ estimates}.
Estimates that include all highest-order norms 
of the solution are said to be of maximal regularity,
which in the case \eqref{abstractevoeq} above means an estimate of type
\begin{align}\label{abstractevoeqmaxreg}
\begin{aligned}
\|\pd_tu\|_{\lr{p}(\BT, X)} + \|Au\|_{\lr{p}(\BT, X)}
+ \|u\|_{\lr{p}(\BT, X)} \leq C\|f\|_{\lr{p}(\BT, X)}. 
\end{aligned}
\end{align}
Such estimates lead to a characterization of $\pd_t+A$ 
as a homeomorphism
in an $\lr{p}(\BT, X)$ setting, which is critical to the analysis 
of non-linear problems.
In the following, we show how to establish periodic $\lr{p}$ estimates of maximal regularity type for a large class of abstract evolution equations based on their $\sR$-solvers,
that is, solution operators of the associated resolvent problems
that satisfy specific $\sR$-bounds.
In particular, we include cases where $0$ lies in the spectrum of $A$,
which constitutes a particular challenge and 
where traditional methods have shortcomings.
As we explain below in more detail, 
in this case the classical maximal regularity estimate \eqref{abstractevoeqmaxreg}
is not available and $\pd_t+A$ can only be a realized as a homeomorphism  
in an adapted framework of function spaces.

Typically, $\lr{p}$ estimates are established via Fourier multipliers,
often via the multiplier theorem of Mikhlin,
which was extended to operator-valued multipliers by \textsc{Weis} \cite{Weis2001}.
He showed that Mikhlin's theorem remains valid in the
operator-valued case if boundedness is replaced 
with $\sR$-boundedness in the assumptions. On the strength of
this result, it is possible to establish maximal regularity 
in $\lr{p}((0, T), X)$ norms for initial-value problems such as
\begin{equation}\label{abstractevoeq_ivp}
\pd_tu + Au = f \quad\text{in $(0, T)$}, \quad u(0) = u_0
\end{equation}
by establishing $\sR$-boundedness on the resolvent family 
\begin{equation}\label{introresolventfamily}
\{\lambda R(\lambda, A) \mid \lambda \in\Sigma_{\theta, \gamma_0}\}, 
\end{equation}
where $R(\lambda, A)=(\lambda {\rm I}-A)^{-1}$ denotes the resolvent operator of $A$ and 
\[
\Sigma_{\theta, \gamma_0} :=\big\{\lambda \in \BC \mid 
{\rm arg}(\lambda) \leq \pi-\theta, \enskip |\lambda| \geq \lambda_0
\big\} \qquad(0 \leq \theta < \frac{\pi}{2})
\]
is a sector, which excludes a ball around the origin if $\lambda_0>0$.
Note that for many problems, 
in particular in unbounded domains,
the inclusion of the origin is not possible
since $0$ does not belong to the resolvent set.
However, for the derivation of $\lr{p}$ estimates for the initial-value problem \eqref{abstractevoeq_ivp}, 
the origin can be excluded from the sector unless $T=\infty$ 
is required.
Since the appearance of \cite{Weis2001}, $\sR$-bounds for resolvent families
of the form \eqref{introresolventfamily} have been established for a
substantial number of boundary-value problems,
which lead to $\lr{p}$ maximal regularity estimates
for the associated initial-value problem. 

In this article, we develop a technique 
to obtain periodic $\lr{p}$ estimates of maximal regularity type
from these $\sR$-bounds. 
In particular, we focus on problems 
where the origin belongs to the spectrum $\sigma(A)$ of $A$, so that 
$\sR$-bounds can at best be established in $\Sigma_{\theta, \gamma_0}$ for some $\gamma_0>0$. 
In the case that $0$ is included in the resolvent set as well as in the $\sR$-bounds, that is, 
the operator family \eqref{introresolventfamily} is $\sR$-bounded for $\gamma_0=0$, 
classical maximal $\lr{p}$ regularity can be
established, which was shown by \textsc{Arendt} and \textsc{Bu} \cite{ArendtBu2002}.
If this is not the case,
classical periodic $\lr{p}$ estimates of maximal regularity type such as \eqref{abstractevoeqmaxreg} cannot be established because 
invertibility of the linear time-periodic problem would require invertibility 
of the associated steady-state problem, that is, that $0$ belongs to the resolvent set $\rho(A)$.
Therefore, we introduce an alternative functional setting in order to characterize 
the parabolic operator $\partial_t+A$ as a homeomorphism with
respect to data in $\lr{p}$ spaces.

The technique developed in the following is based on the transference principle introduced 
by \textsc{de Leeuw} in \cite{dLe65} for scalar-valued multipliers
and generalized to the operator-valued case by 
\textsc{Hyt\"{o}nen}, \textsc{van Neerven}, \textsc{Veraar}, and \textsc{Weis} 
\cite{HytonenVNeervenVeraarWeis2016}.
It states
that $\lr{p}$ boundedness of a continuous Fourier multiplier on $\lr{p}(\BR)$
 is retained when the multiplier
is restricted to $\BZ$ and thus becomes a multiplier in the torus 
$\torus=\BR/2\pi\BZ$ setting. Despite the transference principle
seeing little usage outside the field of harmonic analysis,
we believe it to be an effective tool in the analysis of periodic solutions to
partial differential equations. The promotion of this viewpoint 
is one of the main purposes of this article
since it provides us with an extremely useful tool to derive 
periodic $\lr{p}$ estimates from existing $\sR$-bounds on resolvent
families. If for example the resolvent family \eqref{introresolventfamily} is 
$\sR$-bounded in a sector containing the full imaginary axis $i\BR$,
that is, for $\gamma_0=0$, then the operator-valued version of the 
transference principle combined with the operator-valued version of
Mikhlin's multiplier theorem by \textsc{Weis} \cite{Weis2001} immediately yields 
the periodic $\lr{p}$ estimates \eqref{abstractevoeqmaxreg}.
If, however, $\sR$-bounds are only available in a sector excluding the origin, 
that is, $\gamma_0>0$, a decomposition
technique has to be introduced. Expanding \eqref{abstractevoeq} into a 
Fourier series, we introduce the projection of $u$ into
a lower frequency part $u_{\ell}:=\sum_{|k|\leq\gamma_0} u_k {\mathrm e}^{ikt}$
corresponding to the finite number of modes 
$k\in\BZ$ with $|k|\leq\gamma_0$, and a complementary higher frequency part
 $u_{h}=u-u_{\ell}$.
Based on the $\sR$-bounds of \eqref{introresolventfamily},
we can combine the transference principle with 
Mikhlin's multiplier theorem to establish periodic $\lr{p}$ estimates 
for the higher frequency part.
Provided that a (possibly different) framework of Banach spaces can be 
identified that ensures periodic maximal $\lr{p}$ regularity for the lower frequency part, 
we can combine the two parts
to establish periodic maximal $\lr{p}$ regularity 
for the full problem. Due to the bespoke framework
introduced for the lower frequency part, the resulting type of $\lr{p}$ estimates are not
classical in the sense of \eqref{abstractevoeqmaxreg}.
They are, however, effective in the investigation 
of time-periodic solutions to corresponding nonlinear problems, which we 
demonstrate by specific examples.

Whereas the study of time-periodic solutions 
to ordinary differential equations goes back to the nineteenth
 century, one of the first investigations
of time-periodic partial differential equations is due to \textsc{Prodi} \cite{Prodi1952}, 
who examined the (parabolic) 1D heat equation. Although the work of \textsc{Prodi}
is predated by a few other articles
 \cite{Artemiev1937,Solovieff1939,Karimov1940,Zabotinskij1947}, 
it seems that \cite{Prodi1952} is the first rigorous, 
by contemporary standards, investigation
into the matter. Around the same time, articles also appeared 
on time-periodic solutions to the (hyperbolic) 
wave equation \cite{Prodi1956,FickenFleishman}.
In the following years, the foundation was laid for the methods that 
have nowadays become standard in the study 
of time-periodic partial differential equations.
We shall give a brief overview of the main ideas. 
Consider for this purpose a time-periodic abstract evolution equation
\begin{equation}\label{abstractevoeqOnR}
\pd_tu + Au=F(t, u) \ \ \text{in $\BR$}, \quad u(t+2\pi)=u(t)
\end{equation}
in the classical setting with the whole of $\BR$ as time axis, but still 
considered as equation in a Banach space $X$ for some operator 
$A$ and for $2\pi$-time-periodic (nonlinear) data $F$. 

By far the most popular method that emerged is based on the 
identification of solutions to \eqref{abstractevoeqOnR} 
as fixed points of the so-called Poincar\'e operator%
\footnote{Not to be confused with the \emph{Poincar\'e mapping}, 
which is a related but different notion from the theory of 
dynamical systems.}.
The basic idea goes back to the pioneering work of Poincar\'e 
 \cite{poincare1890,poincare2017} on dynamical systems. 
The Poincar\'e operator, sometimes also referred to as
\emph{translation operator along trajectories}, is the mapping 
$\Phi: X \to X$ that maps an initial value $u_0$ to the value
$u(2\pi)$ of the solution $u$ to the associated initial-value problem
\begin{equation}\label{abstractevoeqOnR_IVP}
\pd_tu + Au=F(t, u) \ \ \text{in $(0,\infty)$}, \quad u(0)=u_0.
\end{equation}
In other words, if $t\mapsto S(t, u_0)$ is the solution operator to \eqref{abstractevoeqOnR}
for the initial value $u_0$, 
the Poincar\'e operator is given by $\Phi(u_0):=S(2\pi, u_0)$.
It is obvious that a fixed point $w_0=\Phi(w_0)$ of $\Phi$ 
is the initial value of a $2\pi$-periodic solution. In this sense
a fixed point of the Poincar\'e operator induces a solution to \eqref{abstractevoeqOnR}.
The main challenge in the application of this method is to construct a setting 
of Banach spaces such that the Poincar\'e operator is well defined
and admits a fixed point. In some cases, this can be carried out directly 
for the nonlinear problem, but often the method is first employed to obtain 
suitable \textit{a priori} estimates of maximal regularity
type for the linearization of \eqref{abstractevoeqOnR}, 
which are subsequently used to investigate the nonlinear problem 
with classical nonlinear functional analysis.
In the context of time-periodic partial differential equations, 
\textsc{Browder} introduced the Poincar\'e operator approach in \cite{Browder_1965},
and around the same time 
\textsc{Krasnosel'ski\u\i} \cite{Krasnoselskii1968} and his student 
\textsc{Kolesov} \cite{Kolesov_1964,Kolesov1966,Kolesov1970} 
advanced the method. 
The investigation of time-periodic solutions 
as fixed points of the Poincar\'e operator depends heavily 
on the framework in which a solution operator to the initial-value problem
\eqref{abstractevoeqOnR_IVP} can be realized.
To illustrate this issue, assume that $A$ generates a 
sufficiently regular semi-group and consider the linear case $F(t, u)=F(t)$. 
The solution operator $S$ then takes the form
\[
S(t, u_0) := \e^{- tA}u_0 
+ \int_0^t \e^{-(t-\tau)A}F(\tau)\, \dd\tau,
\]
and a fixed point $w_0$ of the Poincar\'e operator is therefore given by
\[
w_0=S(2\pi, w_0) \quad\Longleftrightarrow \quad w_0=\bigl({\rm I}-
{\mathrm e}^{-2\pi A}\bigr)^{-1} \int_0^{2\pi} 
{\mathrm e}^{-(2\pi-\tau)A}F(\tau)\, \dd\tau
\]
provided $0\in\rho(A)$, so that $1\in\rho({\mathrm e}^{-2\pi A})$ and 
${\rm I}-\e^{-2\pi A}$ is thereby invertible.
In this case, \textit{a priori} estimates for the time-periodic solution 
$u(t):= S(t, w_0)$ can be established in the setting of Banach spaces 
in which $S$ is realized.
If, on the other hand, $0\in\sigma(A)$, the representation formula 
for $w_0$ above is not valid and it becomes much more difficult to establish 
\textit{a priori} estimates
for the corresponding time-periodic solution. 
For this reason, the Poincar\'e operator approach is seemingly always 
carried out in a setting where $0\in\rho(A)$.
General applications of the method can be found in articles going back 
to \cite{Amann_1978,Pruess1979,VejvodaStrasbraba_1974,VejvodaBook82} for example, 
but also in more recent work such as
 \cite{Lieberman99,BaderKryszewski2003,Cwiszewski2011,Kokocki_2015,NguyenTran2018}.
 More examples can be found in articles devoted to specific equations; so many
that an exhaustive list is beyond the scope of our exposition here. 
We shall mention only the work of 
\textsc{Geissert}, \textsc{Hieber} and \textsc{Nguyen} \cite{GeissertHieberNguyen16}
in which the restriction $0\in \rho(A)$ is 
circumvented by introducing interpolation spaces.

Provided one is able to establish suitable energy estimates for the problem under consideration,
 time-periodic solutions can also be obtained via a Galerkin approximation scheme.
The existence of a time-periodic solution then has to be accomplished in a finite-dimensional
 setting and is thus reduced to finding periodic solutions to an ordinary differential
 equation.
In the finite-dimensional setting, the Poincar\'e operator is compact, and it is therefore much less
 critical to establish existence of a fixed point.
The time-periodic incompressible Navier-Stokes problem 
is a good example of a system that can be treated with
 energy methods; see for example \cite{Prodi1960,GaldiSohr2004}. Also
the time-periodic wave equation with suitable damping 
can be solved in this way \cite{Prouse1964}.
Moreover, time-periodic solutions to the compressible Navier-Stokes equations
can be established,
as was first shown for the one-dimensional case by 
\textsc{Matsumura} and \textsc{Nishida} \cite{MatsumuraNishida_1989}
and then extended 
by several authors, \textit{cf.}~\cite{JinYang_2015,Tsuda_2016} and the references therein.  
Since energy estimates typically lead to \textit{a priori} estimates in Hilbert space settings, this method is not always suited to establish optimal \textit{a priori} estimates for linear 
parabolic problems though.
Nevertheless, it gives a strong tool in the case of hyperbolic or hyperbolic-parabolic mixed systems. 
 
We also want to mention a very different method, which is due to 
\textsc{Seidman} \cite{Seidman_1975}, 
who intentionally avoids using 
the Poincar\'e operator and shows existence of weak time-periodic solutions in 
$\lr{p}(\torus\times\Omega)$ spaces to a nonlinear evolution equation
based on the theory of monotone operators.

A different approach is based on a representation formula that 
arises from the principle that a solution to the initial-value problem
 \eqref{abstractevoeqOnR_IVP} (at least in the linear case $F(t, u)=F(t)$)
 tends to a periodic orbit as $t\to\infty$ regardless of the initial value. 
Equivalently formulated, a solution to the initial-value problem 
with time-periodic right-hand-side
\[
\partial_tu + Au=F(t) \ \ \text{in $(R,\infty)$}, \quad u(R)=u_0, 
\]
tends to a periodic orbit as $R\to -\infty$. Assuming again that
$A$ generates a sufficiently regular semi-group, this principle leads to the formula
\begin{equation}\label{repformulaintfromminfty}
u(t) = \int_\infty^t \e^{-\left(t-\tau\right)A}F(\tau)\,\dd\tau
\end{equation}
for the time-periodic solution.
It is easy to verify that this integral 
expression indeed leads to a periodic solution of the same period as $F$.
As with the Poincar\'e operator approach, the challenge
with the method based on \eqref{repformulaintfromminfty} 
is to construct a framework of Banach spaces 
such that the integral expression is well defined. 
Since $F$ is time-periodic and therefore non-decaying, this clearly requires
suitable decay properties of the semi-group, which again leads 
to $0\not\in \sigma(A)$ as a critical assumption. 
Under this assumption, however, the representation
\eqref{repformulaintfromminfty} can be very useful, 
which was demonstrated already 
in the paper \cite{Prodi1952} by \textsc{Prodi}.
A similar idea was used by \textsc{Fife} \cite{Fife_1964} 
and in subsequent papers 
\cite{Taam_1966,Bange_1975,Nakao_1975,Gaines_1977} 
as well as a number of articles on
specific equations such as the Navier-Stokes equations \cite{KozonoNakao96,Yamazaki2000}.

The principle described above gives rise to yet another approach. 
If namely the solution $u(t)$ to the initial-value problem 
\eqref{abstractevoeqOnR_IVP} tends to a periodic orbit
as $t\to\infty$, then the sequence $u_n(t):=u(t+n2\pi)$, $n\in\BN$, 
will tend to a periodic solution to \eqref{abstractevoeqOnR} 
as $n\to\infty$. This idea was employed in
the context of partial differential equations 
by \textsc{Ficken} and \textsc{Fleishman} \cite{FickenFleishman} as early as 1957, 
and later used to investigate time-periodic solutions to the
Navier-Stokes equations in the incompressible case
\cite{Serrin_PeriodicSolutionsNS1959,Maremonti_TimePer91}
and the compressible case \cite{Valli_1983}.

Finally, we mention the perhaps most natural approach 
to  time-periodic partial differential equations, 
namely the decomposition of data and solution
 into a Fourier series with respect to time.
In the linear case, the investigation then reduces 
to an analysis of the individual Fourier coefficients, 
each of which satisfying a resolvent problem, that is, a time-independent problem; 
see for example
 \cite{Prodi1956,Cesari_1965,Rabinowitz1967,Rabinowitz1969,Hall1970,Brezis_1978}.
This technique, however, has some limitations 
since it is difficult to obtain satisfactory estimates 
of the Fourier series based only on estimates of the
individual coefficients.
Typically, the method only leads to suitable \textit{a priori} estimates 
when working in a framework of absolutely convergent Fourier series,
see the recent articles 
\cite{EiterKyed2021_ViscFlArRigidBodyPerformingTPMotion,Eiter2021_StokesResTPFlowRotating, Eiter2021_OseenResTPFlowRotating} 
for examples from fluid dynamics, or
when Parseval's identity can be invoked,
which requires a Hilbert space setting, 
that is, an investigation in $\lr{2}\left((0,2\pi), H\right)$ for
some Hilbert space $H$.
Our following analysis based 
on Fourier multipliers on the torus $\torus$ offers one way to overcome 
these limitations and to establish \textit{a priori} estimates in
a general $\lr{p}\left((0,2\pi), X\right)$ 
Banach space setting. 

In order to illustrate another significant novelty of our approach, we return to the
 requirement $0\not\in \sigma(A)$ that is needed in both the 
Poincar\'e operator approach and in the method based on the representation formula
 \eqref{repformulaintfromminfty}. The root cause of
this restriction is the necessity in both techniques that the investigation of
 \eqref{abstractevoeqOnR} is carried out in the framework of function spaces of the
 corresponding initial-value problem \eqref{abstractevoeqOnR_IVP}.
Specifically, in both methods the time-periodic solution is characterized as a
 special solution to the initial-value problem,
and can therefore only be estimated in the framework 
in which the initial-value problem is rendered well posed. 
However, this framework is not suitable for \textit{a priori} estimates of solutions to
 the corresponding stationary problem $Au=F$ when $0\in\sigma(A)$.
 Since a stationary solution is trivially also time periodic, it is clear why the
 restriction $0\not\in \sigma(A)$ is imposed.
In our approach, based solely on Fourier multipliers, 
both the Poincar\'e operator and the representation
formula \eqref{repformulaintfromminfty} are avoided, 
and we are able to construct a 
bespoke setting of Banach spaces that enables us
to also treat cases where $0\in\sigma(A)$.

The article is divided into a more theoretical first part (Section \ref{TheorySection}),
where an abstract linear time-periodic problem is investigated, 
and a second part devoted to applications (Sections \ref{sec:bounded_domain} and \ref{sec:ext_domain}).
The theoretical part focuses on 
an abstract time-periodic boundary-value problem, 
and we show
statements in general terms of the periodic maximal $\lr{p}$ 
regularity.
It is based on a combination of the decomposition technique described
 above with the existence of suitable $\sR$-solvers;
 see Theorems \ref{thm:tpprob_abstract} and \ref{thm:tpprob_abstract.hom}.
As examples of the effectiveness of this approach,
we subsequently investigate time-periodic solutions to the 
the $N$-dimensional Navier-Stokes
equations in a periodically moving bounded domain in Section \ref{sec:bounded_domain}, and 
to the three-dimensional Navier-Stokes equations in an exterior domain (at rest) 
in Section \ref{sec:ext_domain}.
Existence of time-periodic solutions to the first problem can be shown in a framework 
of Sobolev spaces,
and in the final results the described decomposition technique is not visible,
which is the case since $0$ is not in the spectrum of the underlying linear operator.
However, for the second problem the situation is different,
and the zero-order mode, which is the time mean of the periodic function,
has to be treated in a separate functional framework.
Moreover, in order to handle the nonlinear terms,
we consider spaces of functions with additional pointwise spatial decay,
and a large part of Section \ref{sec:ext_domain} 
is concerned with the asymptotic properties of solutions.
Since this analysis of the exterior-domain problem is already quite extensive,
the even more involved case of the Navier-Stokes flow inside a periodically moving 
exterior domain is postponed to a future work.

\section{Notation and preliminaries}

\subsection{General notation}

Let $\BN$, $\BZ$, $\BR$ and $\BC$ denote 
 the set of all 
natural numbers, integers, and real and complex numbers, respectively. 
To denote generic constants, we use the symbol $C$, and $C_{a, b, \cdots}$
indicates the dependency of the constant on the quantities
$a$, $b$, $\ldots$. Here, the constants $C$ and $C_{a, b, \cdots}$ may change 
from line to line. 

For any domain $D\subset\BR^N$, $N\in\BN$, 
we denote
Lebesgue spaces, Sobolev spaces, and Besov spaces on $D$
by $\rL_q(D)$, $\rH^m_q(D)$ and $\rB^s_{q,p}(D)$, respectively, 
while $\|\cdot\|_{\rL_q(D)}$, 
$\|\cdot\|_{\rH^m_q(D)}$, and $\|\cdot\|_{\rB^s_{q,p}(D)}$ denote their norms. 
For partial derivatives, we write
$\pd_t = \pd / \pd t$ and $\pd_j = \pd/\pd x_j$. Let $\nabla f = (\pd_1f,
\pd_2f, \pd_3f)$ and $\nabla^2f = (\pd_i\pd_jf \mid
i, j = 1, 2, 3)$. Let $\hat\rH^1_q(D)$ be the
homogeneous space defined by
$$\hat \rH^1_q(D) = \{\varphi \in \rL_{q, {\rm loc}}(D) \mid
\nabla \varphi \in \rL_q(D)^N\}.
$$

For a topological vector space $V$, we let $V'$ denote its dual space. 
In the following, 
$X$ and $Y$ will always denote Banach spaces,
and $\sL(X, Y)$ denotes the space of bounded linear operators from
$X$ to $Y$,
and we simply write $\sL(X)=\sL(X, X)$.
Sometimes, we do not distinguish between a space $X$ and its vector-valued analog $X^N$, 
and we simply write $\|\cdot\|_X$ for the norm of $X^N$.
The set of all $X$-valued holomorphic functions defined
on $U \subset \BC$ is denoted by ${\rm Hol}\,(U, X)$. 
For $\varepsilon\in(0,\pi)$ and $\lambda_0>0$ we define the sectors
$$\Sigma_\varepsilon = \{\lambda \in \BC\setminus\{0\} \mid |\arg\lambda| \leq \pi-\varepsilon\}, 
\quad \Sigma_{\varepsilon, \lambda_0} = \{\lambda \in \Sigma_\varepsilon \mid |\lambda| \geq \lambda_0\}.
$$

\subsection{Time-periodic framework}
The study of partial differential equations in a setting where 
both the data and the corresponding solutions are time periodic
can conveniently be carried out in a framework where the time axis 
is replaced with a torus group. In the following, we consider only the torus
\begin{align*}
\BT:= \BR/2\pi\BZ,
\end{align*}
which provides us with a framework to study $2\pi$-periodic solutions.
We endow $\BT$ with the quotient topology inherited from
$\BR$ via the quotient mapping
\begin{align*}
\pi_{_Q}:\BR\to\BT,\quad\pi_{_Q}(t):=[t] = \{t + 2n\pi \mid n \in \BZ\}.
\end{align*}
Additionally, the quotient mapping induces a differentiable structure on the torus, 
and we can therefore investigate equations
such as \eqref{abstractevoeq} as a differential equation on the smooth manifold $\BT$. 
A solution $u$ in this setting corresponds  
to a classical time-periodic solution 
$u\circ\pi_{_Q}$ in the Euclidean setting and vice versa.
Usually, we tacitly identify functions $u$ on $\BT$ with 
their time-periodic analogue $u\circ\pi_{_Q}$ on $\BR$ .

The topology on $\BT$ turns it into a compact group with a (normalized) 
Haar measure $\dd\tau$ such that
\begin{align*}
\forall u \in C(\BT):\quad \int_\BT u(\tau)\,\dd\tau 
= \frac{1}{2\pi}\int_0^{2\pi} u\circ\pi_{_Q}(t)\,\dd t,
\end{align*}
where $C(\BT)$ is the class of all continuous functions on $\torus$.
Bochner-Lebesgue spaces $\lr{p}(\BT, X)$ for $p\in[1,\infty]$ are then defined in the usual manner.

The differentiable structure gives rise to the space 
\begin{align*}
C^\infty(\BT, X):=
\{ u :\BT\to X \mid  u\circ\pi_{_Q} \in C^\infty(\BR, X)\} 
\end{align*}
of vector-valued smooth functions on $\BT$ 
for any Banach space $X$. The simple structure of $\BT$ implies
that the set of Schwartz-Bruhat functions $\sS(\BT, X)$ (see \cite{Bruhat61,EiterKyed_tplinNS_PiFbook}) 
coincides with the set of smooth functions, that is,
\[
\sS(\BT, X) = C^\infty(\BT, X),
\]
which is endowed with the semi-norm topology induced by the family
$\rho_\ell(u):=\sup_{\tau\in\BT}\|\pd_t^\ell u(\tau)\|_X$, $\ell\in\BN_0$.
We refer to the space 
\[
\sS'(\BT, X) = \sL(\sS(\BT), X)
\]
as the space of $X$-valued tempered distributions on $\BT$. 
One may observe that the notion of classical distributions on
$\BT$ (also known as periodic distributions) 
coincides with the notion of tempered distributions on $\BT$.
Derivatives of distributions $u \in \sS'(\BT, X)$ are defined as distributions 
$\pd_t^\ell u \in \sS'(\BT, X)$
by duality in the usual way.

As a (locally) compact abelian group, the torus $\BT$ 
has a Fourier transform $\sF_\BT$ associated to it.
Obviously, this Fourier transform corresponds to the classic expansion 
of a function on $\BT$ into a Fourier series. 
In this paper, however, it is essential to treat it as Fourier transform 
in the same framework as the Fourier transform $\sF_\BR$ on the real line,
which is defined by
\[
\begin{aligned}
\sF_\BR:
 \sS(\BR, X) &\to \sS(\BR, X), &\quad
\sF_\BR[u](\xi)&:= \frac{1}{2\pi}\int_{\BR}u(x)\,\e^{-i\xi x}\,\dd x, \\
\sF_\BR^{-1}:
 \sS(\BR, X) &\to \sS(\BR, X), &\quad
\sF_\BR^{-1}[v](x)&:= \int_{\BR}v(\xi)\,\e^{i\xi x}\,\dd \xi
\end{aligned}
\]
and extended to mappings $\sF_\BR,\, \sF_\BR^{-1}\colon \sS'(\BR, X') \to \sS'(\BR, X')$ by duality.
To this end, we recall that $\BZ$, endowed with discrete topology 
and counting measure, can be viewed as the dual group of $\BT$.
The $X$-valued Schwartz-Bruhat space on $\BZ$ is given by
\begin{align*}
\sS(\BZ, X) = \{\psi : \BZ \to X \mid \forall\,\ell \in \BN_0 :
\sup_{k \in \BZ}\,|k|^\ell\|\psi(k)\|_X < \infty\}
\end{align*}
and is equipped with the locally convex topology induced by the family of semi-norms
$\hat\rho_\ell$, $\ell\in\BN_0$, 
where $\hat\rho_\ell(\psi):=\sup_{k\in\BZ}\,|k|^\ell\|\psi(k)\|_X$.

In the setting of vector-valued Schwartz-Bruhat spaces, the Fourier transform
 $\sF_\torus$ is the homeomorphism given by
\[
\sF_\torus: \sS(\BT, X) \to \sS(\BZ, X), \quad
\sF_\BT[u](k) := \int_\BT u(t)\e^{-ikt}\,\dd t,
\]
with inverse mapping
\[
\sF^{-1}_\torus: \sS(\BZ, X) \to \sS(\BT, X), \quad
\sF_\BT^{-1}[w](t) := \sum_{k \in \BZ} w(k)\,\e^{ikt}. 
\]
As above, by a duality argument, $\sF_\torus$ extends to a homeomorphism 
on the space of tempered distributions $\sF_\BT: \sS'(\BT, X) \to \sS'(\BZ, X)$
in the usual way.

A standard verification shows that Lebesgue spaces $\lr{p}(\torus, X)$ 
are embedded into $\sS'(\torus, X)$, which enables
us to define vector-valued Sobolev spaces as
\[
\rH^{m}_{p}(\torus, X)
:= \bigl\{u \in\lr{p}(\torus, X) \mid \pd_t^\ell u
 \in\lr{p}(\torus, X) \text{ for }\ell=0,\dots,m\bigr\},
\quad
\|u\|_{\rH^m_p(\torus, X)} 
:= \Bigl(\sum_{\ell=0}^m \|\pd_t^\ell u\|_{\lr{p}(\torus, X)}^p\Bigr)^{\frac1p}
\]
for $m\in\BN$ and $p\in[1,\infty)$. A standard mollification argument shows 
that $\sS(\BT, X) = C^\infty(\BT, X)$ lies dense in $\lr{p}(\BT, X)$ and $\rH^m_p(\BT, X)$. 
For $\theta\in(0,1)$ we further define the fractional Sobolev space $\rH^{\theta}_{p}(\torus, X)$
via Fourier transform by
\begin{align*}
\rH^{\theta}_{p}(\BT, X)
&:= \bigl\{u \in\lr{p}(\BT, X) \mid \sF^{-1}_\BT\bigl[|k|^{\theta} \sF_\BT[u](k)\bigr]
 \in\lr{p}(\BT, X)\bigr\},\\
\|u\|_{\rH^\theta_p(\torus, X)} 
&:= \bigl\| \sF^{-1}_\BT\bigl[(1+|k|)^{\theta} \sF_\BT[u](k)\bigr]\bigr\|_{\lr{p}(\torus, X)}.
\end{align*}

\subsection{$\sR$-boundedness and operator-valued Fourier multipliers}

A family of operators $\CT \subset \sL(X, Y)$ is called $\sR$-bounded in $\sL(X, Y)$
if there exists some $C > 0$
such that for all
$n \in \BN$, $\{T_j\}_{j=1}^n \in \CT^n$, 
and $\{f_j\}_{j=1}^n \in X^n$, we have
\begin{equation}\label{est:Rbound}
\bigl\|\sum_{k=1}^n r_kT_kf_k\bigr\|_{\lr{1}((0,1), Y)} \leq 
C\bigl\|\sum_{k=1}^nr_kf_k\bigr\|_{\lr{1}((0, 1), X)},
\end{equation}
where $r_k$, $k \in \BN$, denote the Rademacher functions given by
$r_k : [0, 1] \to \{-1, 1\}$, $t \mapsto {\rm sign}\,(\sin 2^k\pi t)$. 
The smallest constant $C$ such that \eqref{est:Rbound} holds
is called the $\sR$-bound of $\CT$ and denoted by $\sR_{\sL(X, Y)}\CT$,
If $\CS\subset\sL(X, Y)$ and $\CU\subset\sL(Y,Z)$ are further operator families,
we have
\begin{equation}\label{eq:Rbound.rules}
\begin{aligned}
\sR_{\sL(X, Y)}\{S+T\mid S\in\CS,\,T\in\CT\}
&\leq \sR_{\sL(X, Y)}\CS + \sR_{\sL(X, Y)}\CT,
\\
\sR_{\sL(X, Z)}\{UT\mid U\in\CU,\,T\in\CT\}
&\leq \sR_{\sL(Y, Z)}\CU \cdot \sR_{\sL(X, Y)}\CT;
\end{aligned}
\end{equation}
see \cite[Remark 5.3.14]{HytonenVNeervenVeraarWeis2016} for example.
Due to Kahane's inequality,
one may further replace the spaces $\lr{1}((0, 1), X)$ and $\lr{1}((0, 1), Y)$ in \eqref{est:Rbound}
with $\lr{p}((0, 1), X)$ and $\lr{p}((0, 1), Y)$, respectively,
for any $p\in[1,\infty)$.
In what follows, this choice of $p$ makes no difference.

We introduce the notion of operator-valued Fourier multipliers on $\BR$ and $\BT$.
For $M \in \lr{\infty}(\BR, \sL(X, Y))$ 
we define the operator 
\[
{\rm op}_{\BR}[M]: \sS(\BR, X) \to \sS'(\BR, Y),  \quad
{\rm op}_{\BR}[M]f := \sF_{\BR}^{-1}[M\,\sF_{\BR}[f]].  
\]
For $m \in \lr{\infty}(\BZ, \sL(X, Y))$ we define the operator
\[
{\rm op}_{\torus}[m]: \sS(\BT, X) \to \sS'(\BT, Y), \quad
{\rm op}_\BT[m]f : = \sF^{-1}_\BT[m\,\sF_\BT[f]].
\]
If there exists a continuous extension of ${\rm op}_{\torus}[m]$ to a bounded operator
\begin{align*}
{\rm op}_{\torus}[m]: \lr{p}(\torus, X) \to \lr{p}(\torus, Y),
\end{align*}
we
call $m$ an $\lr{p}(\torus)$-multiplier. 
To identify such $\lr{p}(\torus)$-multipliers,
we shall make use of an operator-valued transference principle,
which relates $\lr{p}(\torus)$-multipliers to $\lr{p}(\BR)$-multipliers,
and combine it with an operator-valued multiplier theorem due to \textsc{Weis} \cite{Weis2001}.
For its formulation, we need the notion of $\sR$-boundedness of operator families introduced above
as well as the notion of UMD spaces.
Recall that a Banach space $X$ is called a UMD space
(or a space of class $\CH\CT$)
if the Hilbert transform $H$ defined by
\[
Hf(t) := \frac{1}{\pi}\lim_{\varepsilon\to0}\int_{|x|\geq\varepsilon} \frac{f(t-s)}{s}\,\dd s, 
\qquad f\in\sS(\BR,X),
\]
extends to a bounded linear operator $\sL(\rL_{p}(\BR,X))$
for $p\in(1,\infty)$. 

Now we can state the multiplier theorem due to \textsc{Weis} \cite[Theorem 3.4]{Weis2001},
which is an operator-valued version of the classical Mikhlin theorem.

\begin{thm}[\textsc{Weis}]\label{WeisMultiplierThm}
Let $X$ and $Y$ be UMD-spaces and $p\in(1,\infty)$. 
Let $M\in\lr{\infty}(\BR, \sL(X, Y))$
be differentiable in $\BR\setminus\{0\}$ and such that
\begin{equation}\label{WeisMultiplierThm_Rbounds}
\sR_{\sL(X, Y)}\bigl\{ M(t) \mid t\in\BR\setminus\{0\}\bigr\} \leq r_0,
\qquad
\sR_{\sL(X, Y)}\bigl\{t M'(t) \mid t\in\BR\setminus\{0\}\bigr\}\leq r_0,
\end{equation}
for some $r_0>0$.
Then
${\rm op}_{\BR}[M]$ extends to a bounded operator
${\rm op}_{\BR}[M]: \lr{p}(\BR, X) \to \lr{p}(\BR, Y)$,
that is, 
$M$ is an $\lr{p}(\BR)$-multiplier, 
and 
\begin{equation}\label{est:WeisMultiplierThm_OpNorm}
\|{\rm op}_{\BR}[M]\|_{\sL(\lr{p}(\BR, X),\lr{p}(\BR, Y))}
\leq C_p\, r_0
\end{equation}
for some constant $C_p>0$ depending only on $p$ but independent of $r_0$.
\end{thm}

In order to investigate $\lr{p}$-boundedness of operators associated with Fourier multipliers
on $\BT$,
we combine this theorem with the following result
that allows an investigation
of $\lr{p}$-boundedness of a Fourier multiplier 
$m \in\lr{\infty}(\BZ, \sL(X, Y))$ in the torus setting
by extending it to a multiplier $M\in\lr{\infty}(\BR; \sL(X, Y))$ in the Euclidean setting.

\begin{thm}[Operator-valued transference principle]\label{VecValuedTransferencePrincipleThm}
Let $X$, $Y$ be Banach spaces and $p \in (1, \infty)$. If
\[
M \in \rL_\infty(\BR, \sL(X, Y)) \cap C(\BR, \sL(X, Y))
\]
is an $\lr{p}(\BR)$-multiplier, that is,
${\rm op}_\BT[M] \in \sL(\rL_p(\BR, X), \rL_p(\BR, Y))$,
then $M_{|\BZ}$ is an $\rL_p(\BT)$-multiplier, that is,
${\rm op}_\BT[M_{|\BZ}] \in \sL(\rL_p(\BT, X), \rL_p(\BT, Y))$, and 
\begin{equation}\label{est:transferenceprinciple}
\|{\rm op}_\BT[M_{|\BZ}]\|_{\sL(\rL_p(\BT, X), \rL_p(\BT, Y))} 
\leq \|{\rm op}_\BR[M]\|_{\sL(\rL_p(\BR, X), \rL_p(\BT, Y))}.
\end{equation} 
\end{thm}

\begin{proof}
This is a special version of \cite[Prop.5.7.1]{HytonenVNeervenVeraarWeis2016},
which is a generalization of the scalar-valued case
originally due to \textsc{de Leeuw} \cite{dLe65}.
\end{proof}

Combining the operator-valued transference principle
with the Weis multiplier theorem,
we directly obtain the following result,
which we employ when studying Fourier multipliers on the torus $\torus$.

\begin{cor}\label{cor:transferenceprinciple}
Let $X$ and $Y$ be UMD spaces, and let
\[
M \in \rL_\infty(\BR, \sL(X, Y))
\cap C(\BR, \sL(X, Y))
\cap C^{1}(\BR\setminus\{0\}, \sL(X, Y))
\]
satisfy \eqref{WeisMultiplierThm_Rbounds} for some $r_0>0$.
Then $M_{|\BZ}$ is an $\rL_p(\BT)$-multiplier, and 
\begin{equation}\label{est:transferenceprinciple_cor}
\|{\rm op}_\BT[M_{|\BZ}]\|_{\sL(\rL_p(\BT, X), \rL_p(\BT, Y))} 
\leq C_p r_0
\end{equation} 
for some constant $C_p>0$ only depending on $p$.
\end{cor}

\begin{proof}
By Theorem \ref{WeisMultiplierThm}, 
$M$ is an $\lr{p}(\BR)$-multiplier,
and Theorem \ref{VecValuedTransferencePrincipleThm}
implies that $M_{|\BZ}$ is an $\lr{p}(\BT)$-multiplier.
Finally, estimate \eqref{est:transferenceprinciple_cor} follows from \eqref{est:transferenceprinciple}
and \eqref{est:WeisMultiplierThm_OpNorm}.
\end{proof}

\section{Maximal $\lr{p}$ regularity for periodic evolution equations}
\label{TheorySection}

We next show how the operator-valued transference principle
 (Theorem \ref{VecValuedTransferencePrincipleThm})
can be utilized to establish periodic $\rL_p$ estimates of the maximal regularity type for a 
large class of abstract linear evolution equations based on their $\sR$-solvers,
that is,
the $\sR$-bounded family of solution operators to the associated resolvent problem.  
We consider the abstract time-periodic boundary-value problem
\begin{equation}\label{eq:tpprob_abstract}
\left\{\begin{aligned}
\pd_tu + Au &=  F &\quad&\text{in $\BT$}, \\
Bu &= TG&\quad&\text{in $\BT$}.
\end{aligned}\right. 
\end{equation}
Here, $A$ is an abstract (differential) operator,
$B$ is a boundary (differential) operator and
$T$ plays the role of a trace operator.
Using this notion, we avoid the introduction of suitable trace classes, 
whose identification strongly relies on the underlying function spaces.
We show that the existence of a unique solution to \eqref{eq:tpprob_abstract}
can be derived from suitable properties of the associated generalized resolvent problem
\begin{equation}\label{eq:resprob_abstract}
\left\{\begin{aligned}
i\sigma w + Aw &= f, \\
Bw &= T g,
\end{aligned}\right. 
\end{equation} 
for $\sigma\in\BR$.
More precisely, while for small $k$ the existence of \textit{a priori} bounds 
in a suitable functional framework is sufficient,
for large $k$ we require the existence of $\sR$-solvers,
that is, a functional framework such that 
the solution operator 
satisfies suitable $\sR$-bounds.
As we shall also see in the examples of the subsequent sections, 
this lack of $\sR$-bounds for the whole line 
often appears in applications.
Moreover, the presented approach allows to use different function spaces for different modes,
which can be a useful and necessary modification;
see also Theorem \ref{thm:tpStokes_extdom} below,
where the zero-order mode requires separate treatment.

\begin{thm}\label{thm:tpprob_abstract} 
Let $X$, $Y$, $Z$ and $W$ be UMD spaces
such that
$X$ and $W$ are continuously embedded into $Y$.
Assume that $A \in \sL(X, Y)$, $B \in \sL(X, Z)$ and $T \in \sL(W, Z)$.
Let 
$\gamma_0\in\BR$, $\beta \in [0, 1]$, and let
$$
\CA \in C^1(\BR\setminus(-\gamma_0, \gamma_0), \sL(Y\times Y \times W, X))
$$
be an operator 
such that for all $\sigma\in\BR\setminus(-\gamma_0, \gamma_0)$
and all $(f, g) \in Y\times W$,
the function $w= \CA(\sigma)(f, \sigma^\beta g, g)$ is the unique solution
to the generalized resolvent problem \eqref{eq:resprob_abstract}.  
Assume the validity of the $\sR$-bounds
\begin{align}
\sR_{\sL(Y\times Y\times W, X)}(\{(\sigma\frac{d}{d\sigma})^\ell \CA(\sigma) \mid
\sigma \in \BR\setminus(-\gamma_0, \gamma_0)\})  &\leq r_0, \label{3.5}\\
\sR_{\sL(Y\times Y\times W, Y)}(\{(\sigma\frac{d}{d\sigma})^\ell (i\sigma\,\CA(\sigma)) \mid
\sigma \in \BR\setminus(-\gamma_0, \gamma_0)\})  &\leq r_0 \label{3.6}
\end{align}
for $\ell=0,1$ and some constant $r_0>0$.

Moreover, for $k \in \BZ$ with $|k|  \leq k_0:=\max\{k\in\BZ\mid k\leq\gamma_0\}$, let $\CX_k$, $\CY_k$, $\CZ_k$ and $\CW_k$
be Banach spaces such that 
$(ik{\rm I}+A) \in \sL(\CX_k, \CY_k)$, $B \in \sL(\CX_k, \CZ_k)$ and $T \in \sL(\CW_k, \CZ_k)$,
and such that for all $(f,g)\in\CY_k\times\CW_k$
there exists a unique solution $w\in\CX_k$ to \eqref{eq:resprob_abstract}
with $\sigma=k$
such that 
\begin{equation}
\label{eq:resest_abstract}
\|w\|_{\CX_k}\leq C_k (\|f\|_{\CY_k}+\|g\|_{\CW_k})
\end{equation}
for some constant $C_k>0$.

Then for any $p\in(1,\infty)$ and $(F, G)$ defined by
\begin{equation}\label{eq:defFG}
F(t)=\sum_{k=-k_0}^{k_0}F_k \e^{ikt} + F_h(t),
\qquad
G(t)=\sum_{k=-k_0}^{k_0}G_k \e^{ikt} + G_h(t),
\end{equation}
with 
$(F_k, G_k) \in \CY_k\times \CW_k$ for $k\in\BZ$, $|k|\leq k_0$,  
and $(F_h, G_h) \in \lr{p}(\BT;Y)\times (\lr{p}(\BT, W)\cap\rH^\beta_p(\BT, Y))$
such that $(\sF_\BT[F_h](k), \sF_\BT[G_h](k))=0$ for all $|k|\leq k_0$,
there exists a unique element
\begin{equation}\label{3.8}
(u_{-k_0},\dots, u_{k_0},u_h) 
\in \CX_{-k_0}\times\dots\times\CX_{k_0} \times ( \rL_p(\BT, X) \cap  \rH^1_p(\BT, Y))
\end{equation}
with $\sF_\BT[u_h](k)=0$ for $|k|\leq k_0$,
such that 
\begin{equation}\label{eq:defu_abstract}
u(t):=\sum_{k=-k_0}^{k_0} u_k \e^{ikt} + u_h(t)
\end{equation}
is the unique solution to the time-periodic problem \eqref{eq:tpprob_abstract},
and
\begin{align}
\label{est:uk}
\|u_k\|_{\CX_k}
&\leq C_k
(\|F_k\|_{\CY_k}+\| G_k\|_{\CW_k})
\\
\label{est:uperp}
\|u_h\|_{\rL_p(\BT, X) \cap \rH^1_p(\BT, Y)}
&\leq C \,r_0(\|F_h\|_{\rL_p(\BT, Y)} + \|G_h\|_{\rH^\beta_p(\BT, Y) \cap \rL_p(\BT, W)}), 
\end{align}
for some constant $C>0$. 
In particular,
\begin{equation}
\label{3.10}\begin{aligned}
\|u\|
&:=\sum_{k=-k_0}^{k_0}\|u_k\|_{\CX_k}
+ \|u_h\|_{\rL_p(\BT, X) \cap \rH^1_p(\BT, Y)}\\
&\leq C\,r_0(\|F_h\|_{\rL_p(\BT, Y)} + \|G_h\|_{\rH^\beta_p(\BT, Y)
 \cap \rL_p(\BT, W)}) 
+ \sum_{k=-k_0}^{k_0}C_k\|(F_k, G_k)\|_{\CY_k\times \CW_k}.
\end{aligned}
\end{equation}
\end{thm}

\begin{proof}
Let $\varphi=\varphi(\sigma)$ be a $C^\infty(\BR)$ function which equals 
$1$ for $|\sigma|\geq (\gamma_0+k_0+1)/2$, 
and equals $0$ for  $|\sigma|\leq \gamma_0$.
We define $u$ as in \eqref{eq:defu_abstract},
where $u_k$ is the unique solution to \eqref{eq:resprob_abstract}
with $\sigma=k$ and $(f,g)=(F_k,G_k)$ for $|k|\leq k_0$, 
and 
\[
u_h = \sF^{-1}_\BT[ \varphi(k)\CA(k)(\sF_\BT[F](k), 
k^\beta \sF_\BT[G](k), \sF_\torus[G](k))].
\]
One readily verifies that $u$ formally satisfies the time-periodic problem \eqref{eq:tpprob_abstract}.
Moreover, for $|k|\leq k_0$
we directly conclude $u_k\in\CX_k$ and estimate \eqref{est:uk}
since $(F_k,G_k)\in\CY_k\times\CW_k$.
To show that $u_h$ also belongs to the claimed function space,
first observe that $\varphi(\sigma)\CA(\sigma)=0$ for $|\sigma|\leq \gamma_0$, 
so that $\varphi\CA \in C^1(\BR, \sL(Y\times Y \times W, X))$.
From \eqref{3.5} and \eqref{3.6} it follows that
\[
\begin{aligned}
\sR_{\sL(Y\times Y\times W, X)}(\{(\sigma\frac{d}{d\sigma})^\ell(\varphi(\sigma) \CA(\sigma)) \mid
\sigma\in\BR\})  
&\leq \|\varphi\|_{\rH^1_\infty(\BR)}r_0, \\
\sR_{\sL(Y\times Y\times W, Y)}(\{(\sigma\frac{d}{d\sigma})^\ell (i\sigma\varphi(\sigma)\CA(\sigma)) \mid
\sigma\in\BR\})  
&\leq \|\varphi\|_{\rH^1_\infty(\BR)}r_0.
\end{aligned}
\]
We can thus apply Corollary \ref{cor:transferenceprinciple} to conclude that 
$k\mapsto\varphi(k) \CA(k)\in\sL(Y\times Y\times W, X)$ 
and $k\mapsto ik\varphi(k) \CA(k)\in\sL(Y\times Y\times W, Y)$
are $\lr{p}(\BT)$-multipliers.
We thus deduce $u_h\in\rL_p(\BT, X) \cap \rH^1_p(\BT, Y)$,
and \eqref{est:uperp} follows from the above $\sR$-bounds
together with \eqref{est:transferenceprinciple_cor}.
In total, we have shown existence of a solution in the asserted function class.

To prove the uniqueness statement, let $u$ be of the form \eqref{eq:defu_abstract}
and satisfy the homogeneous equations, that is,
\eqref{eq:tpprob_abstract} with $(f,g)=0$.
For $k\in\BZ$ with $|k|>k_0$, we set $u_k=\sF_\BT[u_h](k)$.
Then an application of the Fourier transform to \eqref{eq:tpprob_abstract} gives 
\[
\left\{\begin{aligned}
ik u_k + A u_k &= 0, \\
B u_k &= 0 .
\end{aligned}\right. 
\]
We have $u_k\in X$ for $|k| > k_0$
and $u_k\in \CX_k$ for $|k| \leq k_0$,
so that $u_k=0$ follows from the uniqueness assumption for the respective resolvent problem.
We thus conclude $u=0$, 
which completes the proof. 
\end{proof}

The important case of homogeneous boundary conditions can be incorporated in the above setting
in several different manners. 
Of course, the simplest way is to take boundary data $g=0$,
but one may also consider $T=0$ as an (abstract) trace operator.
Another very common way is to simply set $B=0$ and $T=0$, that is,
to drop the abstract boundary condition in \eqref{eq:tpprob_abstract},
and to incorporate the boundary condition in the function space $X$.
Then \eqref{eq:tpprob_abstract} reduces to the
time-periodic $X$-valued ordinary differential equation
\begin{equation}\label{eq:tpprob_abstract.hom}
\pd_tu + Au =  F \quad\text{in $\BT$},
\end{equation}
and \eqref{eq:resprob_abstract} becomes
a proper resolvent problem
\begin{equation}\label{eq:resprob_abstract.hom}
i\sigma w + Aw = f.
\end{equation} 
In this case, the statement of Theorem \ref{thm:tpprob_abstract} also simplifies significantly,
and it can be formulated using the notion of closed linear operators.

\begin{thm}\label{thm:tpprob_abstract.hom} 
Let $Y$ be a UMD space,
and let $A:D(A)\to Y$ be a closed operator with domain $D(A)\subset Y$
and resolvent set $\rho(A)$.
Let 
$\gamma_0\in\BR$
such that
$\{i\sigma\mid \sigma\in\BR\setminus(-\gamma_0,\gamma_0)\} \subset \rho(A)$,
and assume the validity of the $\sR$-bounds
\begin{equation}\label{est:rbound_abstract.hom}
\sR_{\sL(Y)}\big(\big\{\lambda (\lambda{\rm I}-A)^{-1} \mid \lambda=i\sigma,\,
\sigma \in \BR\setminus(-\gamma_0, \gamma_0)\big\}\big)  \leq r_0
\end{equation}
for some constant $r_0>0$.

Moreover, for $k \in \BZ$ with $|k|  \leq k_0:=\max\{k\in\BZ\mid k\leq\gamma_0\}$, 
let $\CX_k$ and $\CY_k$
be Banach spaces such that 
$(ik{\rm I}+A) \in \sL(\CX_k, \CY_k)$
is a linear homeomorphism.

Then for any $p\in(1,\infty)$ and $F$ defined as in \eqref{eq:defFG},
the function $u$ defined in \eqref{eq:defu_abstract}
is the unique solution to the time-periodic problem \eqref{eq:tpprob_abstract.hom},
where $u_k=(ik{\rm I}+A)^{-1} F_k$, and
\begin{equation}
\label{est:uperp.hom}
\|u_h\|_{\rL_p(\BT, X) \cap \rH^1_p(\BT, Y)}
\leq C \,r_0\|F_h\|_{\rL_p(\BT, Y)} 
\end{equation}
for some constant $C>0$. 
\end{thm}

\begin{proof}
We set $X=D(A)\subset Y$, equipped with the usual graph norm,
which is a UMD space since
$(i\gamma_0{\rm I}-A)\colon X\to Y$ is a homeomorphism.
We further set $\CA(\sigma)=(i\sigma{\rm I}+A)^{-1}$ for $|\sigma|\geq\gamma_0$.
Then $\CA(\sigma)$ is a solution operator for the resolvent problem \eqref{eq:resprob_abstract.hom}
for $|\sigma|\geq\gamma_0$,
and we have $\CA \in C^\infty(\BR\setminus(-\gamma_0, \gamma_0), \sL(Y, X))$
by analyticity of the resolvent mapping.
Moreover, due to the identities
\[
i\sigma\CA(\sigma)=i\sigma(i\sigma{\rm I}+A)^{-1},
\qquad
\sigma\frac{d}{d\sigma} (i\sigma\CA(\sigma))
=i\sigma(i\sigma{\rm I}+A)^{-1}+\sigma^2(i\sigma{\rm I}+A)^{-2},
\]
and the formulas from \eqref{eq:Rbound.rules},
the assumed $\sR$-bound \eqref{est:rbound_abstract.hom} directly implies \eqref{3.6}.
Since 
\[
\CA(\sigma)
=\CA(\gamma_0)\CA(\gamma_0)^{-1}\CA(\sigma)
=\frac{\gamma_0-\sigma}{\sigma}\CA(\gamma_0)\,i\sigma(i\sigma{\rm I}+A)^{-1}+{\rm I},
\qquad
\sigma\frac{d}{d\sigma} \CA(\sigma)=i\sigma(i\sigma{\rm I}+A)^{-2},
\]
and since $A(\gamma_0)$ is a homeomorphism in $\sL(Y,X)$,
we infer \eqref{3.5} from  \eqref{est:rbound_abstract.hom} in the same way.
Now the statement directly follows from \ref{thm:tpprob_abstract}
with $W=Z=\CW_k=\CZ_k=\{0\}$ and $B=T=0$. 
\end{proof}

\section{On the Navier-Stokes equations in a bounded periodically moving domain}
\label{sec:bounded_domain}
Let $\Omega$ be a bounded domain in $\BR^N$, $N\geq2$, whose boundary, $\Gamma$, is a compact 
$C^2$ hypersurface.   We assume  that for each  $t \in \BR$
there exists an injective map $\phi(\cdot, t): \Omega \to \BR^N$ such that $\phi(y,0) = 0$ and 
$\phi(y, t+2\pi) = \phi(y, t)$ for $t \in \BR$ 
and $y \in \Omega$; 
possessing the regularity
\begin{equation}
\label{eq:trafo.regularity}
\phi\in
C^0(\torus;C^3(\Omega)^N)\cap
C^1(\torus;C^1(\Omega)^N).
\end{equation}
Again, we identify $2\pi$-periodic on $\BR$ with functions on $\BT$. 
Let  $\Omega_t$ be a domain in $\BR^N$ given 
by setting 
\begin{equation}\label{periodic-domain}
\Omega_t = \{x = y + \phi(y, t) \mid y \in \Omega\} \quad(t \in \BR),
\end{equation} 
that is, $\Omega_t$ is the image of the transformation
$\Phi_t\colon\Omega\to\BR^N$,
$\Phi_t(y)=y+\phi(y,t)$,
for $t\in\BR$.
Notice that $\Omega_t$ is a given bounded periodically moving domain in $\BR^N$
such that $\Omega_{t+2\pi} = \Omega_t$. 
Let $\Gamma_t$ be the boundary of $\Omega_t$, which is given by
$\Gamma_t = \{x = y + \phi(y, t) \mid y \in \Gamma\}$.
We consider the Navier-Stokes equations in $\Omega_t$:
\begin{equation}\label{5.1}
\pd_t\bu + \bu\cdot\nabla\bu - \Delta\bu + \nabla\fp  = \bF,
\quad \dv \bu=0\quad\text{in 
$\Omega_t$}, \quad 
\bu|_{\Gamma_t}  = \bh|_{\Gamma_t}
\end{equation}
for $t \in (0, 2\pi)$. 
Here, $\bu=(u_1, \ldots, u_N)^\top$ is an unknown velocity field, 
$M^\top$ being the transposed $M$, $\fp$ an unknown pressure field,
$\bF = (F_1, \ldots, F_N)^\top$ a prescribed external force,
and $\bh = (h_1, \ldots, h_N)^\top$ a velocity field that prescribes the boundary velocity. 
Assume 
that $\bF(t+2\pi) = \bF(t)$ and $\bh(t+2\pi) = \bh(t)$ for any $t \in \BR$. 
Then system \eqref{5.1} describes the flow of an incompressible viscous fluid
around a periodically moving body,
subject to a time-periodic external force and with prescribed time-periodic boundary conditions.
Note that a natural choice for $\bh$ would be the flow velocity 
associated with the transformation $\Phi_t$,
which means that the fluid particles adhere to the boundary.
This choice corresponds to a no-slip condition,
which we further discuss in Remark \ref{rem:noslip} below.


We transform \eqref{5.1}
into a system in $\Omega$ by using the change of variables induced by $\Phi_t$,
namely
$x = y + \phi(y, t)$. 
For this purpose, we 
assume that
\begin{equation}\label{5.2} 
\|\phi\|_{\rL_\infty(\BT,\rH^3_\infty(\Omega))}
+\|\pd_t\phi\|_{\rL_\infty(\torus,\rH^1_\infty(\Omega))} < \varepsilon_0
\end{equation}
with some small number $\varepsilon_0>0$. 
By this smallness assumption, we may assume the 
existence of the inverse transformation: $y = x + \psi(x, t)$.  The associated Jacobi matrix
$\pd(t, y)/\pd(t, x)$ is given by the formulas:
\[
\frac{\pd t}{\pd t} =1, \quad  \frac{\pd t}{\pd x_j} = 0, \quad
\frac{\pd y_\ell}{\pd t}=\frac{\pd \psi_\ell}{\pd t}, 
\quad \frac{\pd y_\ell}{\pd x_j} = \delta_{\ell j} + \frac{\pd \psi_\ell}{\pd x_j}
\]
for $j, \ell=1, \ldots, N$.  Set $a_{\ell0}(y, t) =(\pd\psi_\ell/\pd t)(y+\phi(y, t), t)$,
and $a_{\ell j}(y, t) = (\pd\psi_\ell/\pd x_j)(y + \phi(y, t), t)$. 
Then partial derivatives transform as
\begin{equation}\label{5.3}
\frac{\pd f}{\pd t} = \frac{\pd g}{\pd t} + \sum_{\ell=1}^Na_{\ell0}(y, t)\frac{\pd g}{\pd y_\ell},
\quad 
\frac{\pd f}{\pd x_j} = \frac{\pd g}{\pd y_j} + \sum_{\ell=1}^Na_{\ell j}(y, t)\frac{\pd g}{\pd y_\ell}
\end{equation}
for $f(x,t)=g(y,t)$.
Set ${\rm J} = \det(\pd x/\pd y) = 1 + {\rm J}_0(y, t)$, which is the Jacobian of $\Phi_t$.
 By the $\rL_\infty$-bounds in \eqref{5.2}
we have 
\begin{equation}\label{5.4}\begin{aligned}
&\sup_{t \in \BR}\|a_{\ell j}(\cdot, t)\|_{\rH^2_\infty(\Omega)}
+ \sup_{t \in \BR}\|\pd_t a_{\ell j}(\cdot, t)\|_{\rL_\infty(\Omega)}
+ \sup_{t \in \BR}\|a_{0j}(\cdot, t)\|_{\rL_\infty(\Omega)} \\
&\qquad +
\sup_{t\in\BR}\|{\rm J}_0(\cdot, t)\|_{\rH^2_\infty(\Omega)} 
+ \sup_{t\in\BR}\|\pd_t{\rm J}_0(\cdot, t)\|_{\rL_\infty(\Omega)} \leq C\varepsilon_0
\end{aligned}\end{equation}
with some constant $C > 0$ for $j, \ell=1, \ldots, N$.  
For  notational simplicity,
we set $\bv(y, t) 
=(v_1, \ldots, v_N)^\top= \bu(x, t)$, and $\fq(y, t)
= \fp(x, t)$.
By \eqref{5.3} we have 
\begin{align}
&\pd_t\bu = \pd_t\bv + \sum_{\ell=1}^N a_{\ell0}\frac{\pd\bv}{\pd y_\ell}, 
\quad 
\bu\cdot\nabla\bu = \bv\cdot({\rm I} + {\rm A})\nabla\bv,  \nonumber \\ 
&\Delta\bu = \Delta\bv + \sum_{\ell=1}^N(a_{\ell j}+ a_{j\ell})\frac{\pd^2\bv}{\pd y_\ell \pd y_j}
+ \sum_{j,\ell, m=1}^N a_{\ell j}a_{mj}\frac{\pd^2\bv}{\pd y_\ell\pd y_m} 
+ \sum_{\ell, m=1}^N\left(\frac{\pd a_{m\ell}}{\pd y_\ell} 
+ \sum_{j=1}^N a_{\ell j}\frac{\pd a_{mj}}{\pd y_\ell}\right)\frac{\pd \bv}{\pd y_m},  \nonumber \\
&\dv\bu 
= {\rm J}^{-1}\{\dv \bv + \dv({\rm J}_0\bv) 
+ \sum_{j,\ell=1}^N\frac{\pd}{\pd y_\ell}(a_{\ell j}{\rm J}v_j)\}. \quad
\nabla\fp = ({\rm I} + {\rm A})\nabla\fq, \label{transform:1}
\end{align}
where ${\rm A}$ is an $(N\times N)$-matrix whose $(j,k)$-th component is
$a_{jk}$.  Setting $w_\ell = v_\ell + {\rm J}_0v_\ell + \sum_{j=1}^N a_{\ell j}{\rm J}v_j$, 
we have ${\rm J}\dv\bu = \dv\bw$ with $\bw = (w_1, \ldots, w_N)^\top$.  Notice that 
$\bw = ({\rm I} + {\rm J}_0{\rm I} + {\rm A}^\top{\rm J})\bv$.
In view of \eqref{5.4}, choosing $\varepsilon > 0$, we see that 
there exists an $(N\times N)$-matrix ${\rm B}_{-1}$ such that 
$({\rm I} + {\rm J}_0{\rm I} + {\rm A}^\top{\rm J})^{-1} = {\rm I}+ {\rm B}_{-1}$
and 
\begin{equation}\label{5.4*}
\sup_{t \in \BR} \|{\rm B}_{-1}(\cdot, t)\|_{\rH^2_\infty(\Omega)} \leq C\varepsilon_0,
\quad \sup_{t \in \BR} \|\pd_t{\rm B}_{-1}(\cdot, t)\|_{\rL_\infty(\Omega)}
\leq C\varepsilon_0.
\end{equation}
Summing up, we see that \eqref{5.1} is transformed to the following equations:
\begin{equation}\label{5.5}
\pd_t\bw - \Delta\bw + \nabla\fq = \bG + \sL(\bw, \fq) + \sN(\bw),
\quad\dv\bw=0 \quad
\text{in $\Omega\times\BT$}, \quad 
\bw|_\Gamma  = \bH|_\Gamma,
\end{equation}
where $\bG$ and $\bH$ are prescribed data and 
\begin{equation}\label{5.6}
\begin{aligned}
\sL(\bw, \fq) & = -\pd_t({\rm B}_{-1}\bw)  
-\sum_{\ell=1}^N a_{\ell0}\frac{\pd}{\pd y_\ell}(({\rm I}+{\rm B}_{-1})\bw)
+ \Delta({\rm B}_{-1}\bw) \\
&
+ \sum_{\ell=1}^N(a_{\ell j} + a_{j\ell})\frac{\pd^2}{\pd y_\ell \pd y_j}(({\rm I} + {\rm B}_{-1})\bw)
+ \sum_{j,\ell, m=1}^N a_{\ell j}a_{mj}\frac{\pd^2}{\pd y_\ell\pd y_m}(({\rm I} + {\rm B}_{-1})\bw) \\
&\quad + \sum_{\ell, m=1}^N\left(\frac{\pd a_{m\ell}}{\pd y_\ell} 
+ \sum_{j=1}^N a_{\ell j}\frac{\pd a_{mj}}{\pd y_\ell}\right)
\frac{\pd }{\pd y_m}(({\rm I} + {\rm B}_{-1})\bw)
- A\nabla\fq
\\
\sN(\bv) & = (({\rm I} + {\rm B}_{-1})\bw)\cdot({\rm I} + {\rm A})\nabla(({\rm I} + {\rm B}_{-1})\bw).
\end{aligned}
\end{equation}
Observe that, due to $\dv\bw=0$ and the boundedness of $\Omega$, 
for the existence of solutions to \eqref{5.5} it is necessary that the boundary data satisfy
\begin{equation}
\label{eq:cond.bdryvel}
\int_{\Gamma}\bH\cdot\bn \,\dd\sigma =0
\end{equation}
for all $t\in\BR$,
where $\bn$ denotes the unit outer normal vector at $\Gamma$.
The following theorem is our main result in this section.  

\begin{thm}\label{thm:periodic.5.1}
Let $2 < p < \infty$ and $N < q < \infty$.
 Then, there exists numbers $\varepsilon,\varepsilon_0 > 0$ such 
that if the condition \eqref{5.2} is valid, and if the prescribed terms $\bG \in 
\rL_p(\BT, \rL_q(\Omega)^N)$ and $\bH\in\rH^1_{p}(\BT, \rL_q(\Omega)^N) \cap 
\rL_{p}(\BT, \rH^2_q(\Omega)^N)$ satisfy the compatibility condition \eqref{eq:cond.bdryvel} 
as well as the smallness condition
\[
\|\bG\|_{\rL_p(\BT, \rL_q(\Omega))}
+\|\bH\|_{\rH^1_{p}(\BT, \rL_q(\Omega))}
+\|\bH\|_{\rL_{p}(\BT, \rH^2_q(\Omega))} \leq \varepsilon^2,
\] 
then 
problem \eqref{5.5} admits a unique\footnote{%
Note that
here and in what follows, the uniqueness statement 
refers to uniqueness of the velocity field, 
but the pressure term is usually merely unique up to a constant.} 
solution $(\bw,\fq)$ with
$$\bw \in \rH^1_{p}(\BT, \rL_q(\Omega)^N) \cap 
\rL_{p}(\BT, \rH^2_q(\Omega)^N),
\quad \fq \in \rL_{p}(\BT, \hat \rH^1_q(\Omega)),
$$
possessing the estimate
$$\|\pd_t\bw\|_{\rL_p(\BT, \rL_q(\Omega))}
+ \|\bw\|_{\rL_p(\BT, \rH^2_q(\Omega))}
+ \|\nabla\fq\|_{\rL_p(\BT, \rL_q(\Omega))} \leq \varepsilon. 
$$
\end{thm}


To prove Theorem \ref{thm:periodic.5.1}, 
we consider the following linearization of equations \eqref{5.5}:
\begin{equation}\label{linear.eq.5.1}
\pd_t\bu - \Delta\bu + \nabla\fp= \bF, \quad\dv\bu=0\quad
\text{in $\Omega\times\BT$}, \quad
\bu|_\Gamma  = \bH|_\Gamma.
\end{equation}
For a maximal-regularity theorem for 
problem \eqref{linear.eq.5.1}, we consider the associated resolvent problem
\begin{equation}\label{5.7}
\lambda\bw - \Delta\bw + \nabla\fp = \bff, \quad\dv\bw=0\quad
\text{in $\Omega$}, \quad 
\bw|_\Gamma  = \bh|_\Gamma.
\end{equation}
At first, we consider the case of homogeneous boundary conditions $\bh=0$,
for which we have the following theorem,
which holds for bounded and exterior domains simultaneously.

\begin{thm}\label{thm:StokesRes} Let $1 < q < \infty$ and $0<\varepsilon < \pi/2$.  
Let $\Omega$ be a bounded domain or an exterior domain in $\BR^N$ $(N \geq 2)$ 
with $C^2$ boundary.  
There exist operator families 
$(\sS(\lambda)) \subset\sL(\rL_q(\Omega)^N,\rH^2_q(\Omega)^N)$
and $(\sP(\lambda)) \subset \sL(\rL_q(\Omega)^N, \hat \rH^1_q(\Omega)))$
such that for every $\lambda\in\Sigma_{\varepsilon}\setminus\{0\}$
and every $\bff \in \rL_q(\Omega)^N$
the pair $(\bw,\fp)=(\sS(\lambda)\bff,\sP(\lambda)\bff)$
is the unique solution to \eqref{5.7} with $\bh=0$
and satisfies the estimate
\begin{equation}\label{est:StokesRes}
|\lambda|\,\| \sS(\lambda)\bff\|_{\rL_q(\Omega)}
+\|\nabla^2 \sS(\lambda)\bff\|_{\rL_q(\Omega)}
+ \|\nabla\sP(\lambda)\bff\|_{\rL_q(\Omega)}
\leq C\|\bff\|_{\rL_q(\Omega)}
\end{equation}
with some constant $C>0$ depending on $\Omega$, $q$ and $\varepsilon$.
Moreover, there exist constants $\lambda_0, r_{0}>0$, 
depending on $\Omega$, $q$ and $\varepsilon$,
such that
$$\sS \in {\rm Hol}\,(\Sigma_{\varepsilon, \lambda_0}, \sL(\rL_q(\Omega)^N, 
\rH^2_q(\Omega)^N)), \quad 
\sP \in {\rm Hol}\,(\Sigma_{\varepsilon, \lambda_0}, \sL(\rL_q(\Omega)^N, 
\hat \rH^1_q(\Omega))),
$$
and
\[
\begin{aligned}
\sR_{\sL(\rL_q(\Omega)^N, \rH^{2-j}_q(\Omega)^N)}
(\{(\lambda\pd_\lambda)^\ell(\lambda^{j/2}\sS(\lambda)) \mid \lambda \in \Sigma_{\varepsilon, \lambda_0}\})
&\leq r_{0}, \\
\sR_{\sL(\rL_q(\Omega)^N, \rL_q(\Omega)^N)}
(\{(\lambda\pd_\lambda)^\ell(\nabla\sP(\lambda)) \mid \lambda \in \Sigma_{\varepsilon, \lambda_0}\})
&\leq r_{0}
\end{aligned}
\]
for $\ell=0,1$, $j=0,1,2$.
Additionally, if $\Omega$ is bounded, then there exist
$\sS(0)\in\sL(\rL_q(\Omega)^N,\rH^2_q(\Omega)^N)$
and $\sP(0)\in\sL(\rL_q(\Omega)^N, \hat \rH^1_q(\Omega))$
such that $(\bw,\fp)=(\sS(0)\bff,\sP(0)\bff)$ 
is the unique solution to \eqref{5.7} for $\lambda=0$
and satisfies
\begin{equation}\label{est:StokesRes.0}
\|\sS(0)\bff\|_{\rH^2_q(\Omega)} + \|\nabla \sP(0)\bff\|_{\rL_q(\Omega)}
\leq C\|\bff\|_{\rL_q(\Omega)}.
\end{equation}
\end{thm}
\begin{proof}
Existence of unique solutions to \eqref{5.7} 
satisfying \eqref{est:StokesRes}
with a uniform constant $C$ for $\lambda\in\Sigma\setminus\{0\}$
was shown in \cite{BorchersSohr1987}.
The analyticity of the associated family of solution operators in $\Sigma_{\varepsilon, \lambda_0}$ 
for some $\lambda_0>0$
together with the asserted $\sR$-bounds
was established in \cite[Theorem 1.6]{Shibata14} 
and \cite[Theorem 9.1.4]{Shibata16}.
Finally, the statement for $\lambda=0$ 
follows from  Fredholm's alternative principle
since the embedding $\rH^2_q(\Omega)\hookrightarrow\rL_q(\Omega)$ is compact 
if $\Omega$ is bounded,
and the solution to \eqref{5.7} is unique.
\end{proof}

As in the results from Section \ref{TheorySection},
in the case of non-zero boundary values, the $\sR$-solvers for the resolvent problem \eqref{5.7} 
are more involved than in the situation of Theorem \ref{thm:StokesRes}.
To state the result, we introduce the space
\[
\CX_q = \{(F_1, F_2, F_3, F_4) \mid F_1, F_2 \in \rL_q(\Omega)^N, 
F_3 \in \rH^1_q(\Omega)^N, F_4 \in \rH^2_q(\Omega)^N\}.
\]
Then we have the following result.

\begin{thm}\label{thm:res.prob.inhom}
Let $1 < q < \infty$ and $0<\varepsilon < \pi/2$.  
Let $\Omega$ be a bounded domain or an exterior domain in $\BR^N$ $(N \geq 2)$ 
with $C^2$ boundary.  
There exist constants $\lambda_0, r_0 > 0$
and operator families
\[
\CS \in {\rm Hol}\,(\Sigma_{\epsilon, \lambda_0}, 
\sL(\CX_q(\Omega), \rH^2_q(\Omega)^N)), \qquad
\CP \in {\rm Hol}\,(\Sigma_{\epsilon, \lambda_0}, 
\sL(\CX_q(\Omega), \hat \rH^1_q(\Omega)))
\]
such that for any $\bff \in \rL_q(\Omega)^N$ and $\bh \in \rH^2_q(\Omega)^N$,
satisfying $\int_\Gamma\bh\cdot\bn\,\dd\sigma=0$ if $\Omega$ is bounded,
the pair $(\bw,\fp)$ defined by
$\bw = \CS(\lambda)(\bff, \lambda\bh, \lambda^{1/2}\bh, \bh)$ and 
$\fp = \CP(\lambda)(\bff, \lambda\bh, \lambda^{1/2}\bh, \bh)$ 
is the unique 
solution to \eqref{5.7}, and 
\[
\begin{aligned}
\sR_{\sL(\CX_q(\Omega), \rH^{2-j}_q(\Omega)^N)}
(\{(\lambda\frac{d}{d\lambda})^\ell(\lambda^{j/2}\CS(\lambda)) \mid
\lambda \in \Sigma_{\epsilon, \lambda_0}\}) &\leq r_0, \\
\sR_{\sL(\CX_q(\Omega), \rL_q(\Omega)^N)}
(\{(\lambda\frac{d}{d\lambda})^\ell(\nabla \CP(\lambda)) \mid
\lambda \in \Sigma_{\epsilon, \lambda_0}\}) &\leq r_0
\end{aligned}
\]
for $\ell=0,1$, $j=0,1,2$.
\end{thm}

\begin{proof}
See \cite[Theorem 1.6]{Shibata14}.
\end{proof}

Now we can prove the following theorem on the time-periodic linear problem
\eqref{linear.eq.5.1}.

\begin{thm}\label{thm:5.2} Let $1 < p, q < \infty$, 
and let $\Omega\subset\BR^N$ be a bounded domain with $C^2$ boundary.  
Then, for any
$\bF \in \rL_{p}(\BT, \rL_q(\Omega)^N)$
and $\bH\in\rH^1_{p}(\BT, \rL_q(\Omega)^N) \cap \rL_{p}(\BT, \rH^2_q(\Omega)^N)$ 
satisfying \eqref{eq:cond.bdryvel}, 
problem \eqref{linear.eq.5.1}
admits a unique solution $(\bu,\fp)$ with
\[
\bu \in \rH^1_{p}(\BT, \rL_q(\Omega)^N) \cap \rL_{p}
(\BT, \rH^2_q(\Omega)^N),\quad 
\fp  \in \rL_{p}(\BT, \hat \rH^1_q(\Omega))
\]
possessing the estimate
\begin{equation}\label{est:thm:5.2}
\begin{aligned}
&\|\pd_t\bu\|_{\rL_p(\BT, \rL_q(\Omega))} + \|\bu\|_{\rL_p(\BT, \rH^2_q(\Omega))}
+ \|\nabla \fp\|_{\rL_p(\BT, \rL_q(\Omega))}
\\ 
&\qquad\qquad
\leq C\big(
\|\bF\|_{\rL_p(\BT, \rL_q(\Omega))}
+\|\pd_t\bH\|_{\rL_p(\BT, \rL_q(\Omega))}
+\|\bH\|_{\rL_p(\BT, \rH^2_q(\Omega))}\big).
\end{aligned}
\end{equation}
\end{thm}

\begin{proof} 
We first consider the case $\bH=0$, for which 
proceed analogously to the proof of Theorem \ref{thm:tpprob_abstract}.
Let $\varphi=\varphi(\sigma)$ be a $C^\infty(\BR)$ function that equals $1$ for $|\sigma|
\geq \lambda_0+1/2$ and $0$ for $|\sigma| \leq \lambda_0+1/4$.
Set 
$$\bu_h = \sF^{-1}_\BT[\sS(ik)\varphi(k)\sF_\BT[\bF](k)],
\quad \fp_h = \sF^{-1}_\BT[\sP(ik)\varphi(k)\sF_\BT[\bF](k)].
$$
Then $\bu_h$ and $\fp_h$ satisfy the equations
$$\pd_t\bu_h - \mu\Delta\bu_h + \nabla\fp_h 
= \bF_h,
\quad \dv\bu_h = 0\quad
\text{in $\Omega\times\BT$}, \quad 
\bu_h|_\Gamma  = 0, 
$$
where we have set $\bF_h = \sF^{-1}_\BT[\varphi(k)\sF_\BT[\bF](k)]$.
Moreover, 
arguing as for the proof of Theorem \ref{thm:tpprob_abstract},
we can use the $\sR$-bounds from Theorem \ref{thm:StokesRes}
and invoke Corollary \ref{cor:transferenceprinciple} to deduce
\begin{equation}\label{est:thm:5.2.1}
\|\pd_t\bu_h\|_{\rL_p(\BT, \rL_q(\Omega))} 
+ \|\bu_h\|_{\rL_p(\BT, \rH^2_q(\Omega))}
+ \|\nabla\fp_h\|_{\rL_p(\BT, \rL_q(\Omega))}
\leq C\|\bF_h\|_{\rL_p(\BT, \rL_q(\Omega))}
\leq C\|\bF\|_{\rL_p(\BT, \rL_q(\Omega))}.
\end{equation}
Now, in view of Theorem \ref{thm:StokesRes}, we set 
\begin{equation}\label{eq:thm:5.2.2}
\bu(t)= \bu_h(t) + \sum_{|k| \leq \lambda_0}\e^{ikt}
\sS(ik)\sF_\BT[\bff](k), \quad
\fp(t) = \fp_h(t) + \sum_{|k| \leq \lambda_0}\e^{ikt}
\sP(ik)\sF_\BT[\bff](k).
\end{equation}
Then,  $\bu$ and $\fp$ satisfy 
\eqref{linear.eq.5.1} with $\bH=0$, 
and from \eqref{est:StokesRes}, \eqref{est:StokesRes.0} and \eqref{est:thm:5.2.1}
we conclude the estimate
\begin{equation}\label{est:thm:5.2.hom}
\|\pd_t\bu\|_{\rL_p(\BT, \rL_q(\Omega))} + \|\bu\|_{\rL_p(\BT, \rH^2_q(\Omega))}
+ \|\nabla \fp\|_{\rL_p(\BT, \rL_q(\Omega))}
\leq C
\|\bF\|_{\rL_p(\BT, \rL_q(\Omega))}.
\end{equation}
Thus, we have shown existence for $\bH=0$.

Now consider arbitrary 
$\bH\in\rH^1_{p}(\BT, \rL_q(\Omega)^N) \cap \rL_{p}(\BT, \rH^2_q(\Omega)^N)$ 
satisfying \eqref{eq:cond.bdryvel}.
Fix $\lambda_1>\lambda_0$ with $\lambda_0$ from Theorem \ref{thm:res.prob.inhom}.
We use the $\sR$-bounded solution operators $\CS$ and $\CP$ from Theorem \ref{thm:res.prob.inhom}
to define
\[
\begin{aligned}
\bu_1 
&= \sF^{-1}_\BT[\CS(ik+\lambda_1)\big(0, (ik+\lambda_1)\tilde\bH_k, 
(ik+\lambda_1)^{1/2}\tilde\bH_k, \tilde\bH_k\big)], 
\\
\fp_1 
&= \sF^{-1}_\BT[\CP(ik+\lambda_1)\big(0, (ik+\lambda_1)\tilde\bH_k, 
(ik+\lambda_1)^{1/2}\tilde\bH_k, \tilde\bH_k\big)],
\end{aligned}
\]
where $\tilde\bH_k=\sF[\bH](k)$.
Then $(\bu_1,\fp_1)$ is a solution to
the auxiliary problem 
$$\pd_t\bu_1+\lambda_1\bu_1 - \mu\Delta\bu_1 + \nabla\fp_1 
= 0,
\quad \dv\bu_1 = 0\quad
\text{in $\Omega\times\BT$}, \quad 
\bu_1|_\Gamma  = \bH,
$$
and by the multiplier theorem from Corollary \ref{cor:transferenceprinciple}, 
we have
\[
\bu_1 \in \rH^1_p(\BT, \rL_q(\Omega)^N) \cap \rL_p(\BT, \rH^2_q(\Omega)^N), 
\quad
\fp_1 \in \rL_p(\BT, \hat \rH^1_q(\Omega))
\]
and the estimate
\[
\|\pd_t\bu_1\|_{\rL_p(\BT, \rL_q(\Omega))}
+ \|\bu_1\|_{\rL_p(\BT, \rH^2_q(\Omega))}
+ \|\nabla \fp_1\|_{\rL_p(\BT, \rL_q(\Omega))} 
\leq C(\|\pd_t\bH\|_{\rL_p(\BT, \rL_q(\Omega))} + \|\bH\|_{\rL_p(\BT, \rH^2_q(\Omega))}).
\]
Here, we have used the interpolation inequality
$$\|\bH\|_{\rH^{1/2}_p(\BT, \rH^1_q(\Omega))} 
\leq C(\|\pd_t\bH\|_{\rL_p(\BT, \rL_q(\Omega))}
+ \|\bH\|_{\rL_p(\BT, \rH^2_q(\Omega))}).
$$
as well as the trivial estimate
\[
\begin{aligned}
\|\sF^{-1}_\BT[(ik+\lambda_1)\tilde \bH_k]\|_{\rL_p(\BT, \rL_q(\Omega))}
&\leq \|\sF^{-1}_\BT[ik\tilde \bH_k]\|_{\rL_p(\BT, \rL_q(\Omega))}
+\lambda_1\|\sF^{-1}_\BT[\tilde \bH_k]\|_{\rL_p(\BT, \rL_q(\Omega))}\\
&=\|\partial_t \bH\|_{\rL_p(\BT, \rL_q(\Omega))}+\lambda_1 \|\bH\|_{\rL_p(\BT, \rL_q(\Omega))}.
\end{aligned}
\]
Now consider the problem
$$\pd_t\bu_2- \mu\Delta\bu_2 + \nabla\fp_2 
= \bF+\lambda_1\bu_1,
\quad \dv\bu_2 = 0\quad
\text{in $\Omega\times\BT$}, \quad 
\bu_2|_\Gamma  = 0.
$$
As shown in the first part of the proof, 
there exists a solution $(\bu_2, \fq_2)$ in the claimed function class
and satisfying \eqref{est:thm:5.2.hom} with $\bF$ replaced with $\bF+\lambda_1\bu_1$.
Then $(\bu, \fq)=(\bu_1+\bu_2, \fq_1+\fq_2)$ 
is a solution to \eqref{linear.eq.5.1},
belongs to the correct function class,
and satisfies estimate \eqref{est:thm:5.2}.

For the uniqueness statement, let $(\bu,\fp)$ be in the considered function class and
satisfy \eqref{linear.eq.5.1} with $\bF=0$ and $\bH=0$.
Then, for each $k \in \BZ$, setting 
$\tilde\bu_k = \sF[\bu](k)$ and $\tilde\fp_k = \sF_\BT[\fp](k)$, 
we see that $\tilde\bu_k \in \rH^2_q(\Omega)^3$ and $\tilde\fp_k 
\in \hat \rH^1_q(\Omega)$ satisfy the homogeneous equations
$$ik\tilde\bu_k - \mu\Delta\tilde\bu_k 
+ \nabla\tilde\fp_k 
= 0,
\quad \dv\tilde\bu_k = 0\quad
\text{in $\Omega$}, \quad 
\tilde\bu_k|_\Gamma  = 0.
$$
Thus, the uniqueness statement from Theorem \ref{thm:StokesRes} yields that
$\tilde \bu_k=\nabla \tilde \fp_k=0$, which shows that
$\bu=\nabla \fp = 0$.  This completes the proof 
of Theorem \ref{thm:osc.1}. 
\end{proof}

\begin{proof}[Proof of Theorem \ref{thm:periodic.5.1}]
We conclude the proof using the contraction mapping principle
in the underlying space $\CI_\varepsilon$ defined by
\begin{align*}
\CI_\varepsilon = \{(\bu, \fp) \mid &\bu \in \rH^1_{p}(\BT, \rL_q(\Omega))
\cap \rL_{p}(\BT, \rH^2_q(\Omega)^N), \
\fp \in \rL_{p}(\BT, \hat \rH^1_q(\Omega)),\\
&E(\bu, \fp) : = \|\pd_t\bu\|_{\rL_p(\BT, \rL_q(\Omega))}
+ \|\bu\|_{\rL_p(\BT, \rH^2_q(\Omega))} + \|\nabla \fp
\|_{\rL_p(\BT, \rL_q(\Omega))} \leq \varepsilon\}.
\end{align*}
Given $(\bu, \fp) \in \CI_\varepsilon$, let $\bv$ and $\fq$ satisfy
\begin{equation}\label{5.12}
\pd_t\bv - \Delta\bv + \nabla\fq = \bG + \sL(\bu, \fp) + \sN(\bu),
\quad \dv \bv=0 \quad
\text{in $\Omega\times\BT$}, \quad
\bv|_\Gamma  = \bH|_\Gamma.
\end{equation}
The existence of $(\bv,\fq)$ follows from Theorem \ref{thm:5.2}
if we can show that the forcing term in \eqref{5.12} belongs to $\rL_p(\BT, \rL_q(\Omega)^N)$.
Firstly, by \eqref{5.2}, \eqref{5.4} and \eqref{5.4*} we have 
\begin{equation}
\label{5.13}
\|\sL(\bu, \fp)\|_{\rL_p(\BT, \rL_q(\Omega))} \leq C\varepsilon E(\bu, \fp).
\end{equation}
In a similar way, noting that $N < q < \infty$, by Sobolev's imbedding theorem we have
\begin{equation}\label{5.14*}
\|\sN(\bu(\cdot, t))\|_{\rL_q(\Omega)} 
\leq C\|\bu(\cdot, t)\|_{\rL_q(\Omega)}
\|\bu(\cdot, t)\|_{\rH^2_q(\Omega)}. 
\end{equation}
By real interpolation theorem, we know that 
\begin{equation}\label{real:int.1}\begin{aligned}
&\rH^1_{p}(\BT, \rL_q(\Omega)) \cap 
\rL_{p}(\BT, \rH^2_q(\Omega)) 
\hookrightarrow C^0(\BT, B^{2(1-1/p)}_{q,p}(\Omega)), \\ 
&\sup_{t \in \BT}\|f(\cdot, t)\|_{B^{2(1-1/p)}_{q,p}(\Omega)}
\leq C(\|f\|_{\rH^1_p(\BT, \rL_q(\Omega))} 
+ \|f\|_{\rL_p(\BT, \rH_q^2(\Omega))})
\end{aligned}\end{equation}
since we have $p > 2$,
and we obtain
\begin{equation}\label{interpolation}
\|f\|_{\rL_\infty(\BT, \rL_q(\Omega))} 
\leq C\|f\|_{\rL_\infty(\BT, \rB^{2(1-1/p)}_{q,p}(\Omega))}
\leq C(\|f\|_{\rH^1_p(\BT, \rL_q(\Omega))}
+ \|f\|_{\rL_p(\BT, \rH^2_q(\Omega))}).
\end{equation}
Therefore,
\begin{equation}\label{5.14}
\|\sN(\bu)\|_{\rL_p(\BT, \rL_q(\Omega))} \leq CE(\bu, \fp)^2.
\end{equation}
Combining estimates \eqref{5.13} and \eqref{5.14} 
with the smallness assumption on $\bG$ and $\bH$, 
choosing $\varepsilon_0\leq C\varepsilon$ and 
applying Theorem \ref{thm:5.2} gives the unique existence of 
a solution $(\bv,\fq)$ of \eqref{5.12} with
$$\bv \in \rH^1_{p}(\BT, \rL_q(\Omega)^N)
\cap \,\rL_{p}(\BT, \rH^2_q(\Omega)^N), \quad 
\fq \in \rL_{p}(\BT, \hat \rH^1_q(\Omega)),
$$
possessing the estimate  
\begin{equation}\label{5.19} 
E(\bv, \fq) \leq C(\varepsilon^2 + \varepsilon E(\bu, \fp)
+ E(\bu, \fp)^2).
\end{equation}
Since we assume that $E(\bu, \fp) \leq \varepsilon$, by \eqref{5.19} we have
$E(\bv, \fq)\leq 3C\varepsilon^2$.
Therefore,  choosing $\varepsilon > 0$ so small that
$3C\varepsilon \leq 1$, we have $E(\bv, \fq) \leq \varepsilon$, so that 
$(\bv, \fq) \in \CI_\varepsilon$. Thus, the map $\Xi$ acting on $(\bu, \fp)
\in \CI_\varepsilon$ by setting $\Xi(\bu, \fp) = (\bv, \fq)$
is a map from $\CI_\varepsilon$ into $\CI_\varepsilon$. 

We next prove that $\Xi$ is a contraction map. Let $(\bu_i, \fp_i)$ $(i=1.2)$ be
any two elements of $\CI_\varepsilon$ and set $(\bv_i, \fq_i) = \Xi(\bu_i, \fp_i)$. 
Since 
\begin{align*}
&\CN(\bu_1)-\CN(\bu_2) \\
&\quad = (({\rm I} + {\rm B}_{-1})(\bu_1-\bu_2))\cdot({\rm I} + {\rm A})
\nabla(({\rm I} + {\rm B}_{-1})\bu_1)
+ (({\rm I} + {\rm B}_{-1})\bu_2)\cdot({\rm I} + {\rm A})\nabla
( ({\rm I} + {\rm B}_{-1})(\bu_1-\bu_2)),
\end{align*}
by Sobolev's inequality and the assumption $N < q < \infty$, we have
$$\|\CN(\bu_1) - \CN(\bu_2)\|_{\rL_q(\Omega))}
\leq C(\|\bu_1-\bu_2\|_{\rL_q(\Omega)}\|\bu_1\|_{\rH^2_q(\Omega)}
+ \|\bu_2\|_{\rL_q(\Omega)}\|\bu_1-\bu_2\|_{\rH^2_q(\Omega)}).
$$
Thus, from \eqref{interpolation} we deduce
\begin{align}
&\|\CN(\bu_1) - \CN(\bu_2)\|_{\rL_p(\BT, \rL_q(\Omega))} \nonumber \\
&\quad \leq C(\|(\bu_1, \bu_2)\|_{\rH^1_p(\BT, \rL_q(\Omega))} + 
\|(\bu_1, \bu_2)\|_{\rL_p(\BT, \rH^2_q(\Omega))})
(\|\bu_1-\bu_2\|_{\rH^1_p(\BT, \rL_q(\Omega))} + 
\|\bu_1-\bu_2\|_{\rL_p(\BT, \rH^2_q(\Omega))}) \nonumber  \\
&\quad  \leq C\varepsilon E(\bu_1-\bu_2, \fp_1-\fp_2).
\label{5.20}
\end{align}
Since $\sL$ is a linear operation, by \eqref{5.13}, we have
\begin{equation}\label{5.21}\begin{aligned}
&\|\sL(\bu_1-\bu_2, \fp_1-\fp_2)\|_{\rL_p(\BT, \rL_q(\Omega))}
\leq C\varepsilon E(\bu_1-\bu_2, \fp_1-\fp_2).
\end{aligned}\end{equation}
Moreover, $\bv=\bv_1-\bv_2$ and $\fq=\fq_1-\fq_2$ satisfy the equations 
$$
\pd_t\bv - \Delta\bv + \nabla\fq = \sL(\bu_1-\bu_2, \fp_1-\fp_2) +
 (\sN(\bu_1)-\sN(\bu_2))\quad \dv\bv=0\quad
\text{in $\Omega\times\BT$}, \quad
\bv|_\Gamma  = 0.
$$
Applying Theorem \ref{thm:5.2} and using \eqref{5.20} and \eqref{5.21}
now gives that 
$$E(\bv_1-\bv_2, \fq_1-\fq_2) \leq C\varepsilon E(\bu_1-\bu_2, \fp_1-\fp_2).
$$
Choosing $\varepsilon>0$ smaller if necessary, we may assume that $C\varepsilon < 1$,
and so $\Xi$ is a contraction map on $\CI_\varepsilon$, which yields 
the unique existence of $(\bu, \fp) \in \CI_\varepsilon$ such that 
$(\bu, \fp) = \Xi(\bu, \fp)$. Obviously, $(\bu, \fp)$ is the required solution
of \eqref{5.5}.  This completes the proof of
Theorem \ref{thm:periodic.5.1}.
\end{proof}

\begin{remark}\label{rem:noslip}
In the case of no-slip boundary conditions, 
the fluid particles at the boundary are attached to the body,
so that the fluid velocity coincides with 
the velocity of boundary particles.
Then, the boundary data $\bh$ in \eqref{5.1}
are given by
\[
\bh(x,t)=\partial_t\Phi_t(y)=\partial_t\phi(t,y),
\]
where $x=\Phi_t(y)=y+\phi(y,t)$,
and in the formulation \eqref{5.5} on a time-independent domain,
this corresponds to boundary data
\[
\bH=({\rm I} + {\rm J}_0{\rm I} + {\rm A}^\top{\rm J})\partial_t\phi.
\]
Therefore, the assumptions on $\bH$ in Theorem \ref{thm:periodic.5.1} are additional 
regularity and smallness assumptions on $\phi$ in this case.
Moreover, the compatibility condition \eqref{eq:cond.bdryvel}
is satisfied if and only if $\Phi_t$ preserves the volume of $\Omega$.
\end{remark}

\section{On periodic Navier-Stokes flow around a body at rest}\label{sec:ext_domain}
\subsection{Problem and main results}\label{subsec:ext_domain.1}
Let $\Omega$ be an exterior domain in $\BR^3$, that is, a domain that is the complement of a compact set. 
We assume that its boundary $\Gamma$ is a $C^2$ hypersurface. 
Let $b > 0$ be a suffciently large number such 
that $\Omega^c \subset B_b$, where $\Omega^c = \BR^3\setminus\Omega$
and $B_b = \{x \in \BR^N \mid |x| < b\}$. 
We further set $S_b=\{x \in \BR^N \mid |x| = b\}$.
We consider the Navier-Stokes equations in $\Omega$:
\begin{equation}\label{eq:NSextdom}
\pd_t\bu + \bu\cdot\nabla\bu - \mu\Delta\bu 
+\nabla\fp  = \bF, \quad \dv\bu = 0\quad\text{in 
$\Omega$}, \quad 
\bu|_{\Gamma}  = \bh|_\Gamma.
\end{equation}
Here, $\pd_j = \pd/\pd x_j$, $x=(x_1, x_2, x_3) \in \BR^3$, 
$\bu=(u_1, u_2, u_3)^\top$ is an unknown velocity field, 
$\fp$ an unknown pressure field, 
$\bF = (F_1, \ldots, F_N)^\top$ a prescribed external force,
and $\bh = (h_1, \ldots, h_N)^\top$ are prescribed boundary data.
Assume that  $\bF(t+2\pi) = \bF(t)$ 
and $\bh(t+2\pi) = \bh(t)$ for any $t \in \BR$.  
Then \eqref{eq:NSextdom} describes the fluid flow around a body,
subject to time-periodic external forcing $\bF$ and with boundary data $\bh$.

Given any time-periodic function $f=f(x, t)$, with period $2\pi$, we write 
\begin{equation}\label{eq:decomposition}
f_S(x)= \int_\BT f(x, t)\,dt=\frac{1}{2\pi}\int^{2\pi}_0 f(x, t)\,dt,
\qquad
f_\perp(x, t) = f(x, t) - f_S(x),
\end{equation}
and $f_S$ and $f_\perp$ are called  the stationary part of $f$ and 
 the oscillatory part of $f$, respectively.  
By means of this decomposition, we divide 
the data and the solution into two parts,
which have different asymptotic properties as $|x|\to\infty$.
To quantify this spatial decay, we set
\[
<f_S>_\ell = \sup_{x \in \Omega} |f_S(x)|
(1 + |x|)^{\ell},
\qquad
<f_\perp>_{p, \ell} = \sup_{x \in \Omega} \|f_\perp(x, \cdot)\|_{\rL_p(\BT)}
(1+|x|)^{\ell}
\]
for $\ell\in\BR$ and $p\in(1,\infty)$.
We shall prove the following theorem.
 
\begin{thm}\label{mainthm:extdom} 
Let $2 < p < \infty$ and   
$3 < q < \infty$. 
Assume that $\bF = \bF_S + \bF_\perp$
with $\bF_S = \dv\bG_S$ and $\bF_\perp = \dv \bG_\perp$. 
Then there exists a 
small constant $\varepsilon>0$ 
depending on $p$ and  $q$ such that if $\bF$ and 
$\bh\in\rH^1_{p}(\BT, \rL_q(\Omega)^3)\cap\rL_{p}(\BT, \rH^2_q(\Omega)^3)$ 
satisfy  the smallness condition
\begin{equation}\label{6:small.2}
<\bF_S>_3 + <\bG_S>_2 
+<\bF_\perp>_{p, 2} + <\bG_\perp>_{p, 1}
+\|\bh\|_{\rH^1_{p}(\BT, \rL_q(\Omega))}
+\|\bh\|_{\rL_{p}(\BT, \rH^2_q(\Omega))}< \varepsilon^2,
\end{equation}
then problem \eqref{eq:NSextdom}  admits a unique solution $(\bu,\fp)$ 
such that
$\bu=\bu_S+\bu_\perp$ and $\fp= \fp_S+\fp_\perp $ with
$$\bu_S \in \rH^2_q(\Omega)^3, \quad
\bu_\perp \in \rH^1_p(\BT, \rL_q(\Omega)) \cap \rL_p(\BT, \rH^2_q(\Omega)), \quad 
\fp_S \in \rH^1_q(\Omega),
\quad \fp_\perp \in \rL_p(\BT, \hat\rH^1_q(\Omega))
$$
satisfying the estimate
\begin{align*}
& <\bu_S>_1 + <\nabla\bu_S>_2 
+ \|\bu_S\|_{\rH^2_q(\Omega)} + \|\fp_S\|_{\rH^1_q(\Omega)}
\\
&\qquad+<\bu_\perp>_{p,1} + <\nabla\bu_\perp>_{p,2}  
+ \|\bu_\perp\|_{\rL_p(\BT, \rH^2_q(\Omega))}
+ \|\pd_t\bu_\perp\|_{\rL_p(\BT, \rL_q(\Omega))}
+ \|\nabla\fp_\perp\|_{\rL_p(\BT, \rL_q(\Omega))}\leq \varepsilon.
\end{align*}
\end{thm}

Our proof of Theorem \ref{mainthm:extdom}
is based on the study of the associated linearized system,
the time-periodic Stokes problem
\begin{equation}\label{eq:tpStokes_extdom}
\pd_t\bv - \mu\Delta\bv + \nabla\fp 
= \bff,
\quad \dv\bv = 0\quad 
\text{in $\Omega\times\BT$}, \quad 
\bv|_\Gamma  = \bh.
\end{equation}
We shall derive the following theorem,
which ensures existence of a unique solution
to \eqref{eq:tpStokes_extdom}
in a framework of spatially weighted spaces.
For shorter notation, 
we set $\rL_{q,3b}(\Omega)=\{f\in\rL_{q}(\Omega)\mid\mathrm{supp}\, f\subset B_{3b}\}$.

\begin{thm}\label{thm:tpStokes_extdom}
Let $1 < p < \infty$, $3<q<\infty$ and $\ell\in(0,3]$.
For all $\bff=\bff_S+\bff_\perp\in \rL_p(\BT, \rL_q(\Omega)^3)$
such that $\bff_S=\dv \bG_S+\bg_S$ and $\bff_\perp=\dv\bG_\perp+\bg_\perp$
with $\bg=\bg_S+\bg_\perp\in\rL_p(\BT, \rL_{q,3b}(\Omega)^3)$ and
\[
<\bG_S>_2 + <\dv\bG_S>_3 + <\bG_\perp>_{p, \ell} + <\dv\bG_\perp>_{p, \ell+1} < \infty,
\]
problem \eqref{eq:tpStokes_extdom}
admits a unique solution $(\bv,\fp)$ with
$$\bv \in \rH^1_p(\BT, \rL_q(\Omega)^3) \cap 
\rL_p(\BT, \rH^2_q(\Omega)^3), \quad
\fp \in \rL_p(\BT, \hat\rH^1_q(\Omega)),
$$
possessing the estimate
\begin{align}
\begin{aligned}
&\|\bv_S\|_{\rH^2_q(\Omega)} +<\bv_S>_1 + <\nabla\bv_S>_2 
+ \|\fp_S\|_{\rH^1_q(\Omega)} + <\fp_S>_2
\\
&\quad
+\|\pd_t\bv_\perp\|_{\rL_p(\BT, \rL_q(\Omega))}
+ \|\bv_\perp\|_{\rL_p(\BT, \rH^2_q(\Omega))}
+<\bv_\perp>_{p, \ell}  
+<\nabla\bv_\perp>_{p, \ell+1}
+ \|\nabla\fp_\perp\|_{\rL_p(\BT, \rL_q(\Omega))}\\
&\qquad\qquad\leq C(
<\dv\bG_S>_3 + <\bG_S>_2
+<\dv\bG_\perp>_{p, 1+\ell} 
+ <\bG_\perp>_{p, \ell} 
\\
&\qquad\qquad\qquad\qquad
+ \|\bg\|_{\rL_p(\BT, \rL_q(\Omega))}
+\|\bh\|_{\rH^1_{p}(\BT, \rL_q(\Omega))}
+\|\bh\|_{\rL_{p}(\BT, \rH^2_q(\Omega))}).
\end{aligned}
\label{est:tpStokes_extdom}
\end{align}
\end{thm} 

For the proof of this theorem,
we split \eqref{eq:tpStokes_extdom}
into two separate problems
by means of the decomposition \eqref{eq:decomposition}.
In the following subsections
these problems are analyzed independently of each other
in the case of vanishing boundary data. 
In Subsection \ref{subsec:ext_domain.final}
we return to the original linear and nonlinear problems \eqref{eq:tpStokes_extdom}
and \eqref{eq:NSextdom}
and complete the proofs of Theorems \ref{mainthm:extdom} and \ref{thm:tpStokes_extdom}.

\subsection{Stationary solutions to the Stokes problem}\label{subsec:ext_domain.stat}
Here we examine time-independent solutions to \eqref{eq:tpStokes_extdom}
with vanishing boundary data $\bh=0$,
that is, solutions $(\bu,\fp)$ to the stationary problem
\begin{equation}\label{eq:statStokes}
-\mu \Delta \bu  + \nabla\fp = \dv \bF + \bg, \quad
\dv \bu = 0 \quad\text{in $\Omega$}, \quad
\bu|_\Gamma = 0.
\end{equation}
We shall derive the following theorem.

\begin{thm}\label{thm:main3} 
Let $3 < q < \infty$. If $\bF$ satisfies the condition  
$<\dv \bF>_3 + <\bF>_2 < \infty$ and $\bg \in \rL_{q,3b}(\Omega)^3$, 
then problem \eqref{eq:statStokes} admits a unique solution
$(\bu,\fp) \in \rH^2_q(\Omega)^3\times\rH^1_q(\Omega)$ possessing the 
estimate:
\begin{equation}\label{est:main3}
\|\bu\|_{\rH^2_q(\Omega)} +<\bu>_1 + <\nabla\bu>_2 
+ \|\fp\|_{\rH^1_q(\Omega)} + <\fp>_2
\leq C(<\dv\bF>_3 + <\bF>_2 + \|\bg\|_{\rL_q(\Omega)})
\end{equation}
with some constant $C > 0$. 
\end{thm}

For the proof, we first consider the Stokes equations in $\BR^3$:
\begin{equation}\label{eq.3.4}
-\mu \Delta \bu + \nabla \fp = \bff, \quad \dv\bu = 0
\quad\text{in $\BR^3$}. 
\end{equation}
As is well-known (cf.~Galdi \cite[pp.239-240]{Galdi}),  there exist fundamental solutions
$\bU=(U_{ij}(x))$ and $\bq = (q_1(x), q_2(x), q_3(x))^\top$
of equations \eqref{eq.3.4} with 
\begin{equation}\label{fund:2}
U_{ij}(x) = -\frac{1}{8\pi\mu}\Bigl(\frac{\delta_{ij}}{|x|} + \frac{x_ix_j}{|x|^3}\Bigr),
\quad 
q_j(x) = \frac{1}{4\pi}\frac{x_j}{|x|^3}.
\end{equation}
If we set 
\begin{equation}\label{fund:1}
\bu(x) =\bU*\bff(x) : =  \int_{\BR^3}\bU(y)\bff(x-y)\,dy,
\quad 
\fp(x) = \bq*\bff : = \int_{\BR^3}\bq(y)\cdot\bff(x-y)\,dy,
\end{equation}
then, $\bu$ and $\fp$ formally satisfy equations \eqref{eq.3.4}. 
We prove the following lemma. 

\begin{lem}\label{lem:3.1} 
Let $3 < q < \infty$.\\
\thetag1~
 Let $\bF$ be a function satisfying 
$<\dv\bF>_3 + <\bF>_2 < \infty$ and set 
$\bu= \bU*(\dv\bF)$ and $\fp = \fq*(\dv\bF)$. 
Then, we have  
\begin{align*}
\|\bu\|_{\rH^2_q(\BR^3)} + \|\fp\|_{\rH^1_q(\BR^3)}
+ <\bu>_1 + <\nabla\bu>_2  + <\fp>_2 \ \leq 
C(<\bF>_2 + <\dv\bF>_3)
\end{align*}
 with some constant $C > 0$.  \\
\thetag2 Let $\bg \in \rL_q(\BR^3)$ such that $\bg$ vanishes for $|x| > b$
with some constant $b>0$.  Let $\bv=\bU*\bg$ and $\fq = \bq*\bg$.
Then, we have 
$$\|\bv\|_{\rH^2_q(\BR^3)} + \|\fq\|_{\rH^1_q(\BR^3)}
+ <\bv>_1+ <\nabla\bv>_2 + <\fq>_2 \ \leq C\|\bg\|_{\rL_q(\BR^3)}
$$
for some constant $C$. 
\end{lem}
\begin{proof} \thetag1~
The theory of singular integrals yields that 
\begin{align*}
\|\nabla^2\bu\|_{\rL_q(\BR^3)} + \|\nabla \fp\|_{\rL_q(\BR^3)}
&\leq C\|\dv\bF\|_{\rL_q(\BR^3)}
\leq \ C_q<\dv\bF>_3;  \\
\|\nabla\bu\|_{\rL_q(\BR^3)} + \|\fp\|_{\rL_q(\BR^3)}
&\leq C\|\bF\|_{\rL_q(\BR^3)} 
\leq \, C_q<\bF>_2.
\end{align*}
For notational simplicity, set $\gamma = <\dv\bF>_3 + <\bF>_2$. 
By  the Gaussian divergence theorem, we write 
\begin{align*}
\bu(x) &= \int_{|y| \leq |x|/2} \bU(y)(\dv\bF)(x-y)\,dy 
- \int_{|y|=|x|/2} \bU(y)\frac{y}{|y|}\cdot\bF(x-y)\,d\omega \\
& + \int_{|x|/2 \leq |y| \leq 2|x|}
\nabla\bU(y)\bF(x-y)\,dy + \int_{|y| \geq 2|x|}\nabla\bU(y)\bF(x-y)\,dy.
\end{align*}
Noting that $|x-y| \geq |x|/2$ for $|y| \leq |x|/2$, $|x-y| \leq 3|x|$ 
for $|x|/2 \leq |y|
\leq 2|x|$, and $|x-y| \geq |y|/2$ for $|y| \geq 2|x|$, 
by \eqref{fund:2} we have
\begin{align*}
|\bu(x)|  \leq C\gamma\Bigl\{ (1+|x|)^{-3}
&\int_{|y| \leq |x|/2}|y|^{-1}\,dy 
+ |x|^{-1}(1+|x|)^{-2} \int_{|y| = |x|/2}\,d\omega  \\
&
+|x|^{-2} \int_{|z| \leq 3|x|}|z|^{-2}\,dz + \int_{|y| \geq 2|x|} |y|^{-4}\,dy\Bigr\}
 \leq C\gamma|x|^{-1}
\end{align*}
for $x\neq0$.
When $|x| \leq 1$, noting that $|(\dv\bF)(x-y)| \leq \gamma$ for $|y| \leq 2$ and 
$|(\dv\bF)(x-y)| \leq C\gamma |y|^{-3}$ for $|y| \geq 2$, we have 
\begin{align*}
|\bu(x)| & \leq \int_{|y| \leq 2}|\bU(y)(\dv\bF)(x-y)|\,dy + \int_{|y| \geq 2}
|\bU(y)(\dv\bF)(x-y)|\,dy \\
&\leq C\gamma\Bigl\{ \int_{|y| \leq 2}|y|^{-1}\,dy 
+ \int_{|y| \geq 2} |y|^{-4}\,dy\Bigr\}\leq C\gamma.
\end{align*}
In total, we thus have $<\bu>_1 \leq C\gamma$. 
In particular, noting that $3 < q < \infty$, we obtain
$$\|\bu\|_{\rL_q(\BR^3)} \leq C_q<\bu>_1 \leq C_q\gamma.$$
Similarly to before, we proceed with the estimate of $\nabla\bu$ and write 
\begin{align*}
\nabla \bu(x) &= \int_{|y| \leq |x|/2} \nabla\bU(y)(\dv\bF)(x-y)\,dy 
- \int_{|y|=|x|/2} \nabla\bU(y)\frac{y}{|y|}\cdot\bF(x-y)\,d\omega
\\
&+ \int_{|x|/2 \leq |y| \leq 2|x|}
\nabla^2\bU(y)\bF(x-y)\,dy + \int_{|y| \geq 2|x|}\nabla^2\bU(y)\bF(x-y)\,dy.
\end{align*}
Then, we have
\begin{align*}
|\nabla\bu(x)|  \leq C\gamma\Bigl\{(1+|x|)^{-3}
&\int_{|y| \leq |x|/2}|y|^{-2}\,dy
+|x|^{-2}(1+|x|)^{-2}\int_{|y|=|x|/2}\,d\omega \\
&+ |x|^{-3} \int_{|z| \leq 3|x|} |z|^{-2}\,dz 
+ \int_{|y| \geq 2|x|}|y|^{-5}\,dy\Bigr\}
\leq C\gamma|x|^{-2}
\end{align*}
for $x\neq 0$. When $|x| \leq 1$, arguing as above, we have 
\begin{align*}
|\nabla\bu(x)| & \leq \int_{|y| \leq 2}|\nabla\bU(y)(\dv\bF)(x-y)|\,dy 
+ \int_{|y| \geq 2}
|\nabla\bU(y)(\dv\bF)(x-y)|\,dy \\
&\leq C\gamma\Bigl\{ \int_{|y| \leq 2}|y|^{-2}\,dy 
+ \int_{|y| \geq 2} |y|^{-5}\,dy\Bigr\}\leq C\gamma.
\end{align*}
Summing up, we have $<\nabla\bu>_2 \leq C\gamma$.

In the very same way, we can use $|\bq(y)|\leq C|y|^{-2}$ and $|\nabla\bq(y)|\leq C|y|^{-3}$ for $y\neq0$
to deduce
$<\fp>_2 \leq C\gamma$.  In particular, 
$\|\fp\|_{\rL_q(\BR^3)} \leq C_q<\fp>_2 \leq C_q\gamma$, because $3 < q < \infty$.
This completes the proof of the first part of Lemma \ref{lem:3.1}. 
\vskip0.5pc\noindent
\thetag2~
As before, the theory of singular integral operators yields that 
\begin{equation}\label{proof:3.1a}
\|\nabla^2\bv\|_{\rL_q(\BR^3)}
+ \|\nabla\fq\|_{\rL_q(\BR^3)}  \leq C\|\bg\|_{\rL_q(\BR^3)}. 
\end{equation}
By estimates for weak singular integral operators
(cf. \cite[II.11]{Galdi}), we further have
\begin{equation}\label{proof:3.1b}
\|\nabla \bv\|_{\rL_q(B_{2b})} + \|\bv\|_{\rL_q(B_{2b})}
+\|\fq\|_{\rL_q(B_{2b})} \leq C_{b,q}\|\bg\|_{\rL_q(\BR^3)}.
\end{equation}
Since $|\bU(x)| \leq C|x|^{-1}$, $|\nabla \bU(x)| \leq C|x|^{-2}$
and $|\bq(x)| \leq C|x|^{-2}$,  noting that 
$\bg(y) = 0$ for $|y| > b$, for $|x| >2b$ we have
\begin{equation}\label{proof:lem.3.1}\begin{aligned}
|\bv(x)| &\leq C\int_{|y|<b}\frac{|\bg(y)|}{|x-y|}\,dy \leq C_b|x|^{-1}
\int_{|y|<b}|\bg(y)|\,\dd y \leq C_b|x|^{-1}\|\bg\|_{\rL_q(\BR^3)}; \\
|\nabla \bv(x)| &\leq C\int_{|y|<b}\frac{|\bg(y)|}{|x-y|^2}\,dy \leq C_b|x|^{-2}
\int_{|y|<b}|\bg(y)|\,\dd y \leq C_b|x|^{-2}\|\bg\|_{\rL_q(\BR^3)}; \\
|\bq(x)| &\leq C\int_{|y|<b}\frac{|\bg(y)|}{|x-y|^2}\,dy \leq C_b|x|^{-2}
\int_{|y|<b}|\bg(y)|\,\dd y \leq C_b|x|^{-2}\|\bg\|_{\rL_q(\BR^3)}.
\end{aligned}\end{equation}
In particular, setting $(B_{2b})^c = \BR^3\setminus B_{2b}$, we conclude 
$$\|\bv\|_{\rL_q((B_{2b})^c)} 
+ \|\nabla\bv\|_{\rL_q((B_{2b})^c)}  + \|\fq\|_{\rL_q((B_{2b})^c)} 
\leq C_{b,q}\|\bg\|_{\rL_q(\BR^3)},
$$
because $3 < q < \infty$, which, combined with \eqref{proof:3.1a} 
and \eqref{proof:3.1b}, yields that 
$$\|\bv\|_{\rH^2_q(\BR^3)} + \|\fq\|_{\rH^1_q(\BR^3)} \leq C_{q,b}
\|\bg\|_{\rL_q(\BR^3)}.$$
By Sobolev's inequality, we have
\begin{align*}
\sup_{|x| \leq 2b}(1+|x|)|\bv(x)| + \sup_{|x| \leq 2b}(1+|x|)^2|\nabla\bv(x)|
+ \sup_{|x| \leq 2b}(1+|x|)^2|\fq(x)|
&\leq C(\|\bv\|_{\rH^2_q(\BR^3)} + \|\fq\|_{\rH^1_q(\BR^3)} )\\
&\leq C_{q,b}
\|\bg\|_{\rL_q(\BR^3)},
\end{align*}
which, combined with \eqref{proof:lem.3.1}, yields that
$$<\bv>_1 + <\nabla\bv>_2 + <\fq>_2 \leq C_{q, b}\|\bg\|_{\rL_q(\BR^3)}.$$
This completes the proof of the second part of Lemma \ref{lem:3.1}. 
\end{proof}

\begin{proof}[Proof of Theorem \ref{thm:main3}] 
To construct a solution operator for problem \eqref{eq:statStokes}, we first consider the 
case where $\bff \in \rL_{q, 3b}(\Omega)^3$. 
Firstly, let $\bff_0$ be the zero extension of 
$\bff$ to the complement of $\Omega$, that is, we set 
$\bff_0(x) = \bff(x)$ for $x \in \Omega$ and $\bff_0(x)=0$ 
for $x \not\in \Omega$. 
Let $\CT_0\bff_0 = \bU*\bff_0$ and $P_0\bff_0 = \bq*\bff_0$. 
Secondly, let
$\bff_b$ be the restriction of $\bff\in \rL_{q, 3b}(\Omega)^3$ to $\Omega_{4b}$,
and let 
$\CA_0$ and $\CB_0$ be the
operators acting 
on $\bff_b \in L_q(\Omega_{4b})^3$  such that 
$\CA_0\bff_b \in\rH^2_q(\Omega_{4b})^3$,
$\CB_0\bff_b \in \hat \rH^1_q(\Omega_{4b})$ 
satisfy the equations
\begin{equation}\label{eq:3.5} 
-\mu\Delta \CA_0\bff_b + \nabla \CB_0\bff_b = \bff_b,  \quad\dv\CA_0\bff = 0
\quad\text{in $\Omega_{4b}$}, \quad
\CA_0\bff_b|_\Gamma = \CA_0\bff_b|_{S_{4b}} = 0, 
\end{equation}
and possess the estimate
\begin{equation}\label{est:3.5}
\|\CA_0\bff_b\|_{\rH^2_q(\Omega_{4b})} + \|\CB_0\bff_b\|_{\hat\rH^1_q(\Omega_{4b})}
\leq C\|\bff_b\|_{\rL_q(\Omega_{4b})}.
\end{equation}
Since $\CB_0\bff_b$ is only defined up to a constant, choosing a constant 
suitably, we may assume that 
\begin{equation}\label{eq:3.6}
\int_{\Omega_{4b}}(P_0\bff_0 - \CB_0\bff_b)\,dx = 0.
\end{equation}
In what follows, let $\varphi$ be a function in  
$C^\infty(\BR^3)$ 
that equals $1$ for $x \in B_{2b}$ and $0$ for $x\not\in B_{3b}$, and let
$\BB$ be the Bogovski\u\i \, operator.  For $\bff \in \rL_{q, 3b}(\Omega)^3$, let  
\begin{equation}\label{proof:3.1}
\CV_0\bff = (1-\varphi)\CT_0\bff_0 + \varphi \CA_0\bff_b + 
\BB[(\nabla\varphi)(\CT_0\bff_0 - \CA_0\bff_b)],
\quad 
\CW_0\bff = (1-\varphi)P_0\bff_0 + \varphi \CB_0\bff_b.
\end{equation}
Inserting these formulas into  equations \eqref{eq:statStokes}, we
have 
\begin{equation}\label{proof:3.2}
-\mu\Delta \CV_0\bff + \nabla\CW_0\bff = \bff + \CR_1\bff, 
\quad \dv \CV_0\bff = 0 \quad\text{in $\Omega$}, \quad \CV_0\bff|_\Gamma=0, 
\end{equation}
where
\[
\begin{aligned}
\CR_1\bff &= 2\mu(\nabla\varphi)\cdot(\nabla\CT_0\bff_0 - \nabla\CA_0\bff_b) 
+ \mu(\Delta\varphi)(\CT_0\bff_0 - \CA_0\bff_b)\\
&\qquad - \mu\Delta\BB[(\nabla\varphi)
\cdot (\CT_0\bff_0 - \CA_0\bff_b)] - (\nabla\varphi)(P_0\bff_0 - \CB_0\bff_b).
\end{aligned}
\]
Employing the same arguments as in \cite{Iwashita} and \cite{KS}, we shall show 
that the inverse operator $({\rm I} + \CR_1)^{-1} \in \sL(\rL_{q, 3b}(\Omega)^3)$ exists
and 
\begin{equation}\label{est:6.3.3}
\|({\rm I} + \CR_1)^{-1} \bff\|_{\rL_q(\Omega)} \leq C\|\bff\|_{\rL_q(\Omega)}
\end{equation}
for any $\bff \in \rL_{q, 3b}(\Omega)^3$.  Postponing proving \eqref{est:6.3.3},
we continue the proof of Theorem \ref{thm:main3}. 
Setting 
\begin{align*}
\CU_0\bff &:= \CV_0({\rm I} + \CR_1)^{-1}\bff \\
&=
(1-\varphi)\CT_0(({\rm I} + \CR_1)^{-1}\bff)_0
+ \varphi \CA_0(({\rm I} + \CR_1)^{-1}\bff)_b \\
&\qquad+ 
\BB[(\nabla\varphi)(\CT_0(({\rm I} + \CR_1)^{-1}\bff)_0 
- \CA_0(({\rm I} + \CR_1)^{-1}\bff))_b],
\\
\CQ_0\bff &:= \CW_0({\rm I} + \CR_1)^{-1}\bff 
= (1-\varphi)P(({\rm I} + \CR_1)^{-1}\bff)_0 
+ \varphi\CB_0(({\rm I} + \CR_1)^{-1}\bff)_b,
\end{align*}
we see that $(\bu,\fp) = (\CU_0\bff,\CQ_0\bff)$ is a solution to problem 
\eqref{eq:statStokes}.  Moreover, by \eqref{est:6.3.3}, 
Lemma \ref{lem:3.1} \thetag2 and \eqref{est:3.5}, we have 
\begin{equation}\label{proof:3.3}\begin{aligned}
\sup_{x\in\Omega}(1+|x|)|\CU_0\bff(x)| &\leq C\|\bff\|_{\rL_q(\Omega)}, &\enskip
\sup_{x \in \Omega}(1+|x|)^2 |\nabla \CU_0\bff(x)| 
&\leq C\|\bff\|_{\rL_q(\Omega)}, \\
\sup_{x \in\Omega}(1+|x|)^2 |\CQ_0\bff(x)| &\leq C\|\bff\|_{\rL_q(\Omega)},
&\quad 
\|\CU_0\bff\|_{\rH^2_q(\Omega)} + \|\CQ_0\bff\|_{\rH^1_q(\Omega)} 
&\leq C\|\bff\|_{\rL_q(\Omega)}.
\end{aligned}\end{equation}

We now consider the case where 
$\bff =\dv\bF + \bg$ with $<\dv\bF>_3+ <\bF>_2 < \infty$
and  $\bg \in \rL_{q, 3b}(\Omega)$.  We write $\bff = \dv((1-\varphi)\bF)+\bh$
with $\bh = \varphi\,\dv\bF + (\nabla\varphi)\cdot\bF + \bg$.  Let  
\begin{align*}
\bu & = (1-\varphi)\CT_0\bff_0 + \BB[(\nabla\varphi)\CT_0\bff_0],
\quad \fp = (1-\varphi)P_0\bff_0.
\end{align*}
Notice that $\bff_0 = \dv((1-\varphi)\bF) + \bh_0$.  We see that 
$\bu$ and $\fp$ satisfy the equations:
$$-\mu\Delta\bu + \nabla\fp = (1-\varphi)\bff_0 +\CR_2\bff, \quad \dv \bu = 0
\quad\text{in $\Omega$}, \quad \bu|_\Gamma = 0,
$$
where we have set
$$\CR_2\bff = 2\mu(\nabla\varphi)\cdot\nabla\CT_0\bff_0 + \mu(\Delta\varphi)
\CT_0\bff_0 -\mu\Delta\BB[(\nabla\varphi)\cdot\CT_0\bff_0]
-(\nabla\varphi)P_0\bff_0.
$$
Moreover, by Lemma \ref{lem:3.1},  we have
\begin{align*}
&\sup_{x \in \Omega}(1+|x|)|\bu(x)| + \sup_{x \in \Omega}(1+|x|)^2|\nabla\bu(x)| + 
\sup_{x \in \Omega}(1+|x|)^2|\fp(x)| + \|\bu\|_{\rH^2_q(\Omega)} 
+ \|\fp\|_{\rH^1_q(\Omega)}\\
&\qquad 
\leq C(<\dv\bF>_3 + <\bF>_2 + \|\bg\|_{\rL_q(\Omega)}).
\end{align*}
Notice that 
$$\|\CR_2\bff\|_{\rL_q(\Omega)} \leq C(<\dv\bF>_3 + <\bF>_2 
+ \|\bg\|_{\rL_q(\Omega)})$$
and 
${\rm supp}\,\CR_2\bff \subset B_{3b}$.  Thus, if we define operators $\CS_0$ and 
$\CP_0$ acting on $\bff$  by setting
\begin{align*}
\CS_0\bff& = (1-\varphi)\CT_0\bff_0 
+ \BB[(\nabla\varphi)\cdot \CT_0\bff_0]
+ \CU_0(\varphi \bff + \CR_2\bff), \\
\CP_0\bff & = (1-\varphi)P_0\bff_0 + \CQ_0(\varphi\bff + \CR_2\bff), 
\end{align*}
then, $\bu=\CS_0\bff$ and $\fp = \CP_0\bff$ satisfy \eqref{eq:statStokes},
and combining the estimates above gives that
\begin{align*}
&\sup_{x\in\Omega}(1+|x|)|(\CS_0\bff)(x)| 
+ \sup_{x\in\Omega}(1+|x|)^2|(\nabla\CS_0\bff)(x)| + 
\sup_{x\in\Omega}(1+|x|)^2|(\CP_0\bff)(x)| \\
&\qquad + \|\CS_0\bff\|_{\rH^2_q(\Omega)} 
+ \|\CP_0\bff\|_{\rH^1_q(\Omega)}
\leq C(<\dv\bF>_3 + <\bF>_2 + \|\bg\|_{\rL_q(\Omega)}), 
\end{align*}
which completes the proof of Theorem \ref{thm:main3}.
\end{proof}
\begin{proof}[Proof of existence of $({\rm I}+\CR_1)^{-1}$]
In what follows, we shall prove 
\eqref{est:6.3.3}.  In view of \eqref{eq:3.6} and
$\nabla(P_0\bff_0 - \CB_0\bff_b) \in \rL_q(\Omega_{4b})^3$, 
we have  $(P_0\bff_0 - \CB_0\bff_b) \in \rH^1_q(\Omega_{4b})$.  
Thus, $\CR_1\bff \in \rH^1_q(\Omega)^3$ and 
${\rm supp}\, \CR_1\bff \subset D_{2b, 3b}$, where $D_{2b, 3b}
= \{x \in \BR^3 \mid 2b \leq |x| \leq 3b\}$. Thus, by Rellich's compactness
theorem, $\CR_1$ is a compact operator on $\rL_{q, 3b}(\Omega)^3$.  Let
${\rm Ker}\, ({\rm I} + \CR_1) = \{ \bff \in \rL_{q, 3b}(\Omega)^3 \mid
({\rm I} + \CR_1)\bff = 0\}$.  By Fredholm's alternative principle, 
if ${\rm Ker}\, ({\rm I} + \CR_1) = \{0\}$, 
then ${\rm I} + \CR_1$ is invertible, and therefore
we have \eqref{est:6.3.3}.  
To verify this, 
we choose $\bff \in {\rm Ker}({\rm I}+\CR_1)$ arbitrarily,
and we shall show that $\bff=0$.  Since $({\rm I}+ \CR_1)\bff=0$, 
we have
$\bff = -\CR_1\bff \in \rH^1_q(\Omega)$ and ${\rm supp}\,\bff \subset D_{2b, 3b}$.
Let $\bu = \CV_0\bff$ and $\fp = \CW_0\bff$. Then by \eqref{proof:3.2}
we see that 
\begin{equation}\label{eq:3.9}
-\mu\Delta\bu + \nabla \fp = 0, \quad\dv\bu=0 \quad\text{in $\Omega$},\quad
\bu|_\Gamma=0.
\end{equation}
Since $3 < q < \infty$ and $\bu \in \rH^2_q(\Omega)^3$
and $\fp \in \hat \rH^1_q(\Omega)$, we have $\bu \in \rH^2_{2, {\rm loc}}(\Omega)$
and $\fp \in \rH^1_{2, {\rm loc}}(\Omega)$.  Let $\psi$ be a $C^\infty(\BR^3)$ 
function which equals $1$ for $|x| < 1$ and $0$ for $|x| > 2$ and set
$\psi_R(x) = \psi(x/R)$ for $R > 4b$.  From \eqref{eq:3.9} it follows
that
\begin{equation}\label{proof:3.4}
0 = (-\mu\Delta \bu + \nabla\fp, \psi_R\bu) = \mu(\nabla\bu, \psi_R\nabla\bu)
+\mu(\nabla\bu, (\nabla\psi_R)\bu) - (\fp, (\nabla\psi_R)\cdot\bu).
\end{equation}
Using Lemma \ref{lem:3.1} \thetag2, we obtain
$$|\bu(x)| \leq C|x|^{-1}, \quad|\nabla\bu(x)| \leq C|x|^{-2}, \quad 
\quad 
|\fp| \leq C|x|^{-2}
$$
for $|x| > 4b$, and so we have 
\begin{align*}
&|(\nabla\bu, (\nabla\psi_R)\bu)| 
\leq C\|\nabla\psi\|_{L_\infty(\BR^3)}R^{-1}\int_{2R \leq |x| \leq 3R}|x|^{-3}
\,\dd x
= O(R^{-1}) \to 0, \\
&|(\fp, (\nabla\psi_R)\bu)| 
\leq C\|\nabla\psi\|_{L_\infty(\BR^3)}R^{-1}\int_{2R \leq |x| \leq 3R}|x|^{-3}
\,\dd x
= O(R^{-1}) \to 0
\end{align*}
as $R\to \infty$, and so taking $R\to\infty$ in \eqref{proof:3.4}, we have
$\|\nabla\bu\|_{\rL_2(\Omega)} = 0$, which implies that $\bu$ is a constant 
vector.  But, $\bu|_\Gamma=0$, and so $\bu=0$.  Thus, by the first equation
of \eqref{eq:3.9},
$\nabla \fp=0$, which shows that $\fp$ is a constant.  But,
$\fp(x) = O(|x|^{-2})$ as $|x| \to \infty$, and so $\fp=0$. Therefore,
by \eqref{proof:3.1} we have
\begin{equation}\label{proof:3.5} \begin{aligned}
&(1-\varphi)\CT_0\bff_0 + \varphi\CA_0\bff_b
+ \BB[(\nabla\varphi)\cdot(\CT_0\bff_0 - \CA_0\bff_b)]=0, \\
&(1-\varphi)P_0\bff_0 + \varphi \CB_0\bff_b=0
\end{aligned}\end{equation}
in $\Omega_{4b}$. 
Since $\BB[(\nabla\varphi)\cdot(\CT_0\bff_0 - \CA_0\bff_b)]$ 
vanishes for $x \not \in D_{2b, 3b}$ and $\varphi(x) = 0$ for 
$|x| > 3b$ and $1-\varphi(x)=0$ for $|x| < 2b$, we have
\begin{equation}\label{proof:3.6}
\CA_0\bff_b= 0, \enskip \CB_0\bff_b = 0 \quad\text{for $|x| < 2b$},
\quad \CT_0\bff_0 = 0, \enskip P_0\bff_0 = 0\quad\text{for $|x| > 3b$}.
\end{equation}
Let
$$\bw(x) = \begin{cases} (\CA_0\bff_b)(x) &\quad\text{for $x \in \Omega_{4b}$}, \\
0 &\quad\text{for $x \not\in \Omega$}, 
\end{cases}
\quad 
\fq(x) = \begin{cases} (\CB_0\bff_b)(x) &\quad\text{for $x \in \Omega_{4b}$}, \\
0 &\quad\text{for $x \not\in \Omega$},
\end{cases}
$$
and then, by  \eqref{est:3.5} and \eqref{proof:3.6} $\bw \in \rH^2_q(B_{4b})^3$
and $\fq \in \rH^1_q(B_{4b})$, and $\bw$ and $\fq$ satisfy equations:
\begin{equation}\label{proof:3.7}
-\mu\Delta \bw + \nabla \fq = \bff_0, \quad \dv \bw = 0 
\quad\text{in $B_{4b}$}, \quad \bw|_{S_{4b}}=0.
\end{equation}
On the other hand, by \eqref{proof:3.6}, we know that 
$\CT_0\bff_0$ and $P_0\bff_0$ also satisfy equations \eqref{proof:3.7},
and so the uniqueness of solutions yields that 
$\bw = \CT_0\bff_0$ and $\nabla(\fq-P_0\bff_0) = 0$ in $B_{4b}$. 
Noting that $\fq=\CB_0\bff_b$, by \eqref{eq:3.6} we have 
$\fq = P_0\bff_0$.  In particular, $(\nabla\varphi)\cdot(\CT_0\bff_0-
\CA_0\bff_b) = 0$. Thus, from \eqref{proof:3.5} we even have
\begin{align*}
0 &= \CT_0\bff_0 - \varphi(\CT_0\bff_0 - \CA_0\bff_b) = \CT_0\bff_0, \\
0 &= P_0\bff_0 - \varphi(P_0\bff_0 - \CB_0\bff_b) = P_0\bff_0,
\end{align*}
which gives that
$\bff = -\mu\Delta \CT_0\bff_0 + \nabla P_0\bff_0 = 0$ in $\Omega_{4b}$. 
Thus, we have $\bff = 0$.  This completes the proof of existence of 
$({\rm I}+\CR_1)^{-1}$.  
\end{proof}


\subsection{Purely oscillatory solutions to the Stokes problem}\label{subsec:ext_domain.osc}

In this section we consider the oscillatory part of the linearization \eqref{eq:tpStokes_extdom}
for $\bh=0$, 
which is given by
\begin{equation}\label{eq:poStokes_extdom}
\pd_t\bv_\perp - \mu\Delta\bv_\perp + \nabla\fp_\perp 
= \bff_\perp,
\quad \dv\bv_\perp = 0\quad 
\text{in $\Omega\times\BT$}, \quad 
\bv_\perp|_\Gamma  = 0,
\end{equation}
where the subscript $\perp$ indicates that all functions have vanishing time mean.
We shall prove the following theorem.

\begin{thm} \label{thm:osc.1} 
Let $1 < p, q < \infty$. Then, 
for any $\bff_\perp \in \rL_p(\BT, \rL_q(\Omega))$
with $\int_\BT \bff_\perp(\cdot, s)\,\dd s=0$, problem \eqref{eq:poStokes_extdom}
admits a unique solution $(\bv_\perp,\fp_\perp)$ with
$$\bv_\perp \in \rH^1_p(\BT, \rL_q(\Omega)^3) \cap 
\rL_p(\BT, \rH^2_q(\Omega)^3), \enskip 
\fp_\perp \in \rL_p(\BT, \hat\rH^1_q(\Omega)),
\enskip \int_\BT\bv_\perp(\cdot, s)\,\dd s= 0, 
\enskip \int_\BT\fp_\perp(\cdot, s)\,\dd s=0,
$$
possessing the estimate
\begin{equation}\label{est:osc.1}
\|\pd_t\bv_\perp\|_{\rL_p(\BT, \rL_q(\Omega))}
+ \|\bv_\perp\|_{\rL_p(\BT, \rH^2_q(\Omega))}
+ \|\nabla\fp_\perp\|_{\rL_p(\BT, \rL_q(\Omega))}
\leq C\|\bff_\perp\|_{\rL_p(\BT, \rL_q(\Omega))}.
\end{equation}
\end{thm}

In order to prove Theorem \ref{thm:osc.1},
we first consider the corresponding resolvent problem
\begin{equation}\label{eq:StokesRes_extdom}
\lambda \bw - \mu \Delta \bw + \nabla\fr = \bff,
\quad \dv \bw=0
\quad\text{in $\Omega$}, \quad 
\bw|_\Gamma  = 0,
\end{equation}
for which we have the results from Theorem \ref{thm:StokesRes}.
Observe that, with regard to the time-periodic problem \eqref{eq:tpStokes_extdom},
we are mainly interested in the resolvent problem \eqref{eq:StokesRes_extdom}
with $\lambda=ik$ for $k\in\BZ$,
and Theorem \ref{thm:StokesRes} gives a framework where this problem
is uniquely solvable, but merely for $k\in\BZ\setminus\{0\}$.
This is the main reason why we only consider the purely oscillatory problem \eqref{eq:poStokes_extdom}
in Theorem \ref{thm:osc.1}.
Apart from this, Theorem \ref{thm:osc.1}
can be proved in the same way as Theorem \ref{thm:5.2}.

\begin{proof}[Proof of Theorem \ref{thm:osc.1}]
Let $\lambda_0$ be as in Theorem \ref{thm:StokesRes},
and let $\varphi=\varphi(\sigma)$ be a $C^\infty(\BR)$ function that equals $1$ for $|\sigma|
\geq \lambda_0+1/2$ and $0$ for $|\sigma| \leq \lambda_0+1/4$.
Using the operator families $\sS$ and $\sP$ from Theorem \ref{thm:StokesRes}, we set
$$\bv_h = \sF^{-1}_\BT[\sS(ik)\varphi(k)\sF_\BT[\bH_\perp](k)],
\quad \fp_h = \sF^{-1}_\BT[\sP(ik)\varphi(k)\sF_\BT[\bH_\perp](k)].
$$
Then $\bv_h$ and $\fp_h$ satisfy the equations
$$\pd_t\bv_h - \mu\Delta\bv_h + \nabla\fp_h 
= \bH_h,
\quad \dv\bv_h = 0\quad
\text{in $\Omega\times\BT$}, \quad 
\bv_h|_\Gamma  = 0, 
$$
where we have set $\bH_h = \sF^{-1}_\BT[\varphi(k)\sF_\BT[\bH_\perp](k)]$.
Moreover, 
arguing as in the proof of Theorem \ref{thm:tpprob_abstract},
we can use the $\sR$-bounds from Theorem \ref{thm:StokesRes}
and employ Corollary \ref{cor:transferenceprinciple} to deduce
\begin{equation}\label{osc:est.1}
\|\pd_t\bv_h\|_{\rL_p(\BT, \rL_q(\Omega))} 
+ \|\bv_h\|_{\rL_p(\BT, \rH^2_q(\Omega))}
+ \|\nabla\fp_h\|_{\rL_p(\BT, \rL_q(\Omega))}
\leq C\|\bH_h\|_{\rL_p(\BT, \rL_q(\Omega))}
\leq C\|\bH_\perp\|_{\rL_p(\BT, \rL_q(\Omega))}.
\end{equation}
Now, in view of Theorem \ref{thm:StokesRes}, we set 
\[
\bv_\perp(t)= \bv_h(t) + \sum_{0 < |k| \leq \lambda_0}\e^{ikt}
\sS(ik)\sF_\BT[\bH_\perp](k), \quad
\fp_\perp(t) = \fp_h(t) + \sum_{0 < |k| \leq \lambda_0}\e^{ikt}
\sP(ik)\sF_\BT[\bH_\perp](k).
\]
Then,  $\bv_\perp$ and $\fp_\perp$ satisfy equations 
\eqref{eq:poStokes_extdom}, 
and from \eqref{est:StokesRes} and \eqref{osc:est.1}
we conclude estimate
\eqref{est:osc.1}.
Thus, we have shown the existence part of Theorem \ref{thm:osc.1}.
The uniqueness statement follows exactly as in the proof of Theorem \ref{thm:5.2}
noting that $\sF_\torus[\bv_\perp](0)=0$ and $\sF_\torus[\fp_\perp](0)=0$ by assumption.
\end{proof}


Next we examine the pointwise decay of 
the solution $(\bv_\perp,\fp_\perp)$.
More precisely, we show decay properties of 
$\|\bv_\perp(x, \cdot)\|_{\rL_p(\BT)}$
with respect to the $x$-variable,
as stated in the following theorem. 

\begin{thm}\label{thm:main.4.1} 
In the situation of Theorem \ref{thm:osc.1},
let $3<q<\infty$ and $\ell\in(0,3]$
such that $\bff_\perp= \dv \bF_\perp+ \bg_\perp$ with
\begin{equation}\label{assump:4.1}\begin{aligned}
&\int_\BT\bF_\perp(x, t)\,\dd t=0,
&\quad &\!\!\!<\bF_\perp>_{p, \ell} + <\dv\bF_\perp>_{p, \ell+1}< \infty,  \\
&\int_\BT\bg_\perp(x, t)\,\dd t=0, &\quad 
&\bg_\perp \in \rL_p(\BT, \rL_{q, 3b}(\Omega)).
\end{aligned}\end{equation}
Then,  $\bv_\perp$ has the following asymptotics:
\begin{equation}\label{eq:4.2}\begin{aligned}
<\bv_\perp>_{p, \ell}  + 
<\nabla\bv_\perp>_{p, \ell+1} 
\leq C(<\dv\bF_\perp>_{p, \ell+1} 
+ <\bF_\perp>_{p, \ell}  + \|\bg_\perp\|_{\rL_p(\BT, \rL_q(\Omega))})
\end{aligned}\end{equation}
with some constant $C > 0$. 
\end{thm}

\begin{remark} 
Since $3 < q < \infty$, we have
$\|\dv\bF_\perp\|_{\rL_p(\BT, \rL_q(\Omega))} \leq C_q<\dv\bF_\perp>_{p, \ell+1}$
and so 
\begin{equation}\label{est:4.1}
\|\bff_\perp\|_{\rL_p(\BT, \rL_q(\Omega))}
\leq C_q(<\dv\bF_\perp>_{p, \ell+1} + \|\bg_\perp\|_{\rL_p(\BT, \rL_q(\Omega))}).
\end{equation}
Therefore, Theorem \ref{thm:osc.1}
really shows existence for $\bff_\perp$ as in Theorem \ref{thm:main.4.1}.
\end{remark}

To prove \eqref{eq:4.2}, we use the following theorem due to 
\textsc{Eiter} and \textsc{Kyed} \cite{EK1},
which collects properties of the velocity fundamental solution $\Gamma_\perp$ to 
\eqref{eq:poStokes_extdom},
which is a tensor field $\Gamma_\perp$ 
such that $\bv_\perp:=\Gamma_\perp\ast \bH_\perp$ is formally a solution to \eqref{eq:poStokes_extdom} 
for $\Omega=\BR^3$.
\begin{thm}\label{lem:4.1} 
Let 
\begin{equation}\label{kernel}
\Gamma_\perp = \sF^{-1}_{\BR^3\times\BT}\Bigl[\frac{1-\delta_\BZ}{\mu|\xi|^2
+ ik}\Bigr({\rm I} - \frac{\xi\otimes\xi}{|\xi|^2}\Bigr)\Bigr].
\end{equation}
Then, it holds $\Gamma_\perp \in \rL_q(\BR^3\times\BT)^{3\times3}$ for $q \in (1, 5/3)$,
and $\pd_j\Gamma_\perp  \in \rL_q(\BR^3\times\BT)$ for $q \in (1, 5/4)$, $j = 1, 2, 3$.
Moreover, for any $\alpha \in \BN_0^3$, $\delta>0$ and  $r \in [1, \infty)$, 
there exists a constant $C_{\alpha,\delta} > 0$ such that
$$
\forall |x|\geq\delta:\quad
\|\pd_x^\alpha\Gamma_\perp(x, \cdot)\|_{\rL_r(\BT)} 
\leq \frac{C_{\alpha,\delta}}{|x|^{3+|\alpha|}}.
$$
\end{thm}
\begin{remark} This theorem holds for any dimension $N\geq 2$ replacing 
$5/3$, $5/4$ and $3+|\alpha|$ with $(N+2)/N$, $(N+2)/(N+1)$ and $N+|\alpha|$, respectively. 
\end{remark}

\begin{proof}[Proof of Theorem \ref{thm:main.4.1}]
Since we assume that
$3 < q < \infty$, by Sobolev's inequality,
we have 
$$\sup_{|x| \leq 4b} \|\bv_\perp(\cdot, x)\|_{\rL_p(\BT)} 
+ \sup_{|x| \leq 4b} \|(\nabla \bv_\perp)(\cdot, x)\|_{\rL_p(\BT)} 
\leq C\|\bv_\perp\|_{\rL_p(\BT, \rH^2_q(\Omega))}
\leq C\|\bH_\perp\|_{\rL_p(\BT, \rL_q(\Omega))}.
$$
It thus remains to estimate $\bv_\perp$ for $|x| > 4b$. 
To this end, recall the operator families
$\sS$ and $\sP$ given in Theorem \ref{thm:StokesRes}. 
As seen in the proof of Theorem \ref{thm:osc.1}, 
we have $\bv_\perp = \sF^{-1}_\BT[\sS(ik)\sF_\BT[\bH_\perp](k)]$ and $\fp_\perp
= \sF^{-1}_\BT[\sP(ik)\sF_\BT[\bH_\perp](k)]$.
We shall first give a representation formula of 
$\sS(ik)$ for $k \in \BZ\setminus\{0\}$ for $|x| > 3b$, which will be used to 
investigate the asymptotic behavior of $\bv_\perp$ for $|x| > 3b$. 
Notice that 
$\sS(ik) \in \sL(\rL_q(\Omega)^3, \rH^2_q(\Omega)^3)$ and 
$\sP(ik) \in \sL(\rL_q(\Omega)^3, \hat \rH^1_q(\Omega))$ 
satisfy
the estimate
\begin{equation}\label{est:6.4}
 \|\sS(ik)\bff\|_{\rH^2_q(\Omega)} + \|\nabla \sP(ik) \bff
\|_{\rL_q(\Omega)} \leq C\|\bff\|_{\rL_q(\Omega)}
\end{equation}
for $\bff\in\rL_q(\Omega)^3$,
where 
$C$ depends solely on $q$ and $\Omega$. 
Moreover, $\bu=\sF_\BT[\bv_\perp](k)=\sS(ik)\sF_\BT[\bH_\perp](k)$ and $\fq = 
\sF_\BT[\fp_\perp](k)=\sP(ik)\sF_\BT[\bH_\perp](k)$ satisfy the equations
\begin{equation}\label{eq:perp.1}
ik\bu - \mu\Delta\bu + \nabla \fq = \bff_k, 
\quad \dv\bu= 0 \quad\text{in $\Omega$},
\quad \bu|_\Gamma=0,
\end{equation}
where $\bff_k = \sF_\BT[\bff_\perp](k)$.
Let $\varphi$ be a function in $C^\infty_0(\BR^3)$ that equals $1$ for $|x| < 2b$
and $0$ for $|x| > 3b$.  Let  
\begin{equation}\label{proof:4.2}
\bw= (1-\varphi)\sS(ik)\bff_k + \BB[(\nabla\varphi)\cdot
\sS(ik)\bff_k], \quad 
\fr = (1-\varphi)\sP(ik)\bff_k.
\end{equation}
Then $\bw \in \rH^2_q(\BR^3)^3$ and $\fr \in \hat \rH^1_q(\BR^3)$.
Moreover, by \eqref{eq:perp.1}  $\bw$ and $\fr$ satisfy the equations
$$ik\bw - \mu\Delta\bw + \nabla \fr = 
(1-\varphi)\bff_k + \CR_3(ik)\bff_k, \quad \dv \bw = 0
\quad\text{in $\BR^3$},
$$
where we have set 
\begin{equation}\label{eq:2.7*} \begin{aligned}
\CR_3(\lambda)\bff &= 2\mu(\nabla\varphi)\cdot\nabla \sS(\lambda)\bff
+\mu(\Delta\varphi)\sS(\lambda)\bff - (\nabla \varphi)\sP(\lambda)\bff
+(\lambda-\mu\Delta)\BB[(\nabla\varphi)\cdot
\sS(\lambda)\bff].
\end{aligned}\end{equation}
By the uniqueness of solutions to the Stokes resolvent problem in $\BR^3$, we have
$\bw = \CT(ik)((1-\varphi)\bff_k + \CR_3(ik)\bff_k)$,
where
\begin{equation}\label{eq:2.3}
\CT(\lambda)\bff = \sF^{-1}_{\BR^3}\Bigl[\frac{1}
{\mu|\xi|^2 + \lambda}\Bigr({\rm I} - \frac{\xi\otimes\xi}{|\xi|^2}\Bigr)\sF_{\BR^3}[\bff]\Bigr].
\end{equation}
Since $1-\varphi = 1$ 
and $\BB[(\nabla\varphi)\cdot\CS(ik)\bff_k]=0$ 
for $|x| >4b$, by \eqref{proof:4.2} we thus have 
\begin{equation}\label{eq:2.7}
\sS(ik)\bff_k = \CT(ik)((1-\varphi)\bff_k)
+ \CT(ik)(\CR_3(ik)\bff_k)
 \quad(|x| > 4b)
\end{equation}
for any $k \in \BZ\setminus\{0\}$.
 Thus, we have
\begin{equation}\label{eq:perp.2}\begin{aligned}
\bv_\perp &= \sF^{-1}_\BT[(1-\delta_\BZ(k))\sS(ik)\sF_\BT[\bH_\perp](k)]
\\
&= \sF^{-1}_\BT[(1-\delta_\BZ(k))\CT(ik)\sF_\BT[
(1-\varphi)\bH_\perp](k))]\\
&\qquad
+ \sF^{-1}_\BT[(1-\delta_\BZ(k))\CT(ik)(\CR_3(ik)
\sF_\BT[\bH_\perp](k))]
\end{aligned}\end{equation} 
for $|x| > 4b$. 
Moreover, from Theorem \ref{thm:StokesRes} we conclude 
\begin{gather}
\sR_{\sL(\rL_q(\Omega)^3, \rH^1_q(\BR^3)^3)}(\{(\lambda\pd_\lambda)^\ell
\CR_3(\lambda) \mid 
\lambda \in \BR\setminus[-\lambda_0, \lambda_0]\}) \leq r_0
\quad(\ell=0,1), \label{eq:StokesRes_extdom2-1}\\
\| \CR_3(ik)\bff_k\|_{\rH^1_q(\BR^3)} \leq r_0\|\bff_k\|_{\rL_q(\Omega)}
\label{eq:StokesRes_extdom2-2}
\end{gather}
for any  $k\in\BZ\setminus\{0\}$  
with some constant 
 $r_0$. 
In particular, we define $\CR_{4}\bH_\perp$ by setting 
$\CR_{4}\bH_\perp= \sF^{-1}_\BT[
(1-\delta_\BZ(k))\CR_3(ik)\bff_k]$.
Then, employing Corollary \ref{cor:transferenceprinciple}
in the same way as in the proof of Theorem \ref{thm:osc.1}, 
we see that 
\begin{equation}\label{eq:StokesRes_extdom4}\begin{aligned}
&{\rm supp}\, \CR_4\bH_\perp \subset D_{2b, 3b} 
: = \{(x, t) \in \BR^3 \times \BR \mid 2b < |x| < 3b\}, \\
&\qquad\|\CR_4\bH_\perp\|_{\rL_p(\BT, \rL_q(\Omega))} 
\leq C\|\bH_\perp\|_{\rL_p(\BT, \rL_q(\Omega))}.
\end{aligned}
\end{equation}
Recalling that $\bH_\perp = \dv\bF_\perp + \bg_\perp$, we set 
 $\bG = (1-\varphi)\bF_\perp$
and $\bh =  
(\nabla\varphi)\bF_\perp + (1-\varphi)\bg_\perp +
\CR_{4}\bH_\perp$.  
In virtue of \eqref{kernel}, \eqref{eq:2.3} and \eqref{eq:perp.2}, we then have
\begin{equation}\label{eq:perp.3}\begin{aligned}
\bv_\perp &= \Gamma_\perp*(\dv\bG) + \Gamma_\perp*\bh \\
&= \int_\BT\int_{\BR^3}\Gamma_\perp(y,s)(\dv\bG)(x-y, t-s)\,\dd y \dd s
+ \int_\BT\int_{\BR^3}\Gamma_\perp(y,s)\bh(x-y, t-s)\,\dd y \dd s.
\end{aligned}\end{equation}
for $|x| > 4b$.  Set $\bv_1 = \Gamma_\perp*(\dv\bG)$ and
$\bv_2 = \Gamma_\perp*\bh$. 
By the divergence theorem of Gau\ss,   we write
\begin{align*}
&\bv_1(x,t) = \nabla \Gamma_\perp * \bG(x,t)\\
 & = \int_\BT\int_{|y| \leq 1}\nabla\Gamma_\perp(y, s)
\bG(x-y, t-s)\,\dd y\dd s + \int_\BT\int_{1 \leq |y| \leq |x|/2}
\nabla\Gamma_\perp(y, s)\bG(x-y, t-s)\,\dd y\dd s \\
&+\int_\BT\int_{|x|/2\leq |y| \leq 2|x|} \nabla\Gamma_\perp(y,s)
\bG(x-y, t-s)\,\dd y\dd s 
+\int_\BT\int_{ |y| \geq 2|x|} \nabla\Gamma_\perp(y,s)
\bG(x-y, t-s)\,\dd y\dd s.
\end{align*}
Let $r_0$ and $r_1$ be exponents such that $p < r_0 < \infty$, $r_1 \in (1, 5/4)$
and $1 + 1/r_0 = 1/r_1 + 1/p$.  Then, we have Young's inequality
\begin{equation}\label{conv:1}
\|f*g\|_{\rL_{r_0}(\BT)} \leq \|f\|_{\rL_{r_1}(\BT)} \|g\|_{\rL_p(\BT)}
\end{equation}
for $f*g(t) = \int_\BT f(s)g(t-s)\,\dd s$. 
Setting $\gamma_\ell = <\bG>_{p,\ell}$, from Theorem \ref{lem:4.1} we thus conclude 
\begin{align*}
\|\bv_1(x, \cdot)\|_{\rL_{r_0}(\BT)} 
&\leq \gamma_\ell\|\nabla\Gamma_\perp\|_{\rL_{r_1}(B_1\times \BT)}(1+ |x|)^{-\ell}
+C_1\gamma_\ell\int_{1\leq |y| \leq |x|/2}|y|^{-4}\,dy (1+|x|)^{-\ell}\\
& \qquad+ C_1\gamma_\ell(|x|/2)^{-4}\int_{|z| \leq 3|x|}(1 + |z|)^{-\ell}\,dz 
+ C_1\gamma_\ell\int_{|y| \geq 2|x|}|y|^{-4-\ell}\,dy.
\end{align*} 
Noting that $p \leq r_0$ and $\gamma_\ell \leq \ <\bF_\perp>_{p, \ell}$, we infer 
\begin{align*}\| \bv_1(x, \cdot)\|_{\rL_p(\BT)} &\leq C_b|x|^{-\min\{\ell,4\}}<\bF_\perp>_{p, \ell}
\quad\text{for $|x| \geq 4b$}.
\end{align*}
Analogously, we write 
\begin{align*}
\nabla\bv_1(x,t)  &= \int_\BT\int_{|y| \leq 1}\nabla\Gamma_\perp(y, s)
(\dv \bG)(x-y, t-s)\,dyds \\
&+ \int_\BT\int_{1 \leq |y| \leq |x|/2}
\nabla\Gamma_\perp(y, s)(\dv\bG)(x-y, t-s)\,dyds \\
&+\int_\BT\int_{|x|/2\leq |y| \leq 2|x|} \nabla\Gamma_\perp(y,s)
(\dv\bG)(x-y, t-s)\,dyds  \\
&+\int_\BT\int_{ |y| \geq 2|x|} \nabla^\ell\Gamma_\perp(y,s)
(\dv\bG)(x-y, t-s)\,dyds.
\end{align*}
Setting $\gamma_{\ell+1} = <\dv\bG>_{p, \ell+1}$, 
by Theorem \ref{lem:4.1} and \eqref{conv:1} we have 
\begin{align*}
\|\nabla\bv_1(x, \cdot)\|_{\rL_{r_0}(\BT)}
&\leq \gamma_{\ell+1}\|\nabla\Gamma_\ell\|_{\rL_{r_1}(B_1\times\BT)}(1+ |x|)^{-\ell-1}
+C_1\gamma_{\ell+1}\int_{1\leq |y| \leq |x|/2}|y|^{-4}\,dy (1+|x|)^{-\ell-1}\\
&\qquad + C_1\gamma_{\ell+1}(|x|/2)^{-4}\int_{|z| \leq 3|x|}(1 + |z|)^{-\ell-1}\,dz 
+ C_1\gamma_{\ell+1}\int_{|y| \geq 2|x|}|y|^{-5-\ell}\,dy.
\end{align*}
Since we have
$$<\dv\bG>_{p, \ell+1} \leq\  <\dv \bF_\perp>_{p, \ell+1} + <(\nabla\varphi)
\bF_\perp>_{p, \ell+1}\leq\  <\dv \bF>_{p, \ell+1} + \|\nabla\varphi\|_{\rL_\infty(\BR^3)}
3b<\bF>_{p, \ell}
$$
and $p \leq r_0$, we thus obtain 
\begin{align*}\|\nabla\bv_1(x, \cdot)\|_{\rL_p(\BT)} &\leq C_b|x|^{-\min\{\ell+1,4\}}
(<\dv\bF_\perp>_{p, \ell+1} + <\bF_\perp>_{p, \ell})
\quad \text{for $|x| \geq 4b$}. 
\end{align*}
Finally, we
use that $\bh(y,s)$ vanishes for $|y|\geq 3b$.
For $m=0,1$ we thus have
\[
\nabla^m\bv_2(x,t)
= \int_\BT\int_{|x-y| \leq 3b}\nabla^m\Gamma_\perp(y, s)
\bh(x-y, t-s)\,dyds
\]
Since $|x|\geq 4b$ and $|x-y|\leq 3b$ implies $|y|\geq |x|/4\geq b$,
by Theorem \ref{lem:4.1} and \eqref{conv:1}, we deduce 
\begin{align*}
\|\nabla^m\bv_2(x, \cdot)\|_{\rL_{p}(\BT)}
&\leq  \int_{|x-y| \leq 3b}\|\nabla^m\Gamma_\perp(y,\cdot)\|_{\rL_p(\BT)}
\|\bh(x-y, \cdot)\|_{\rL_1(\BT)}\,dy
\leq  C_m |x|^{-3-m}\|\bh\|_{\rL_1(B_{3b}\times\BT)}.
\end{align*}
Noting \eqref{eq:StokesRes_extdom4},
we can estimate the last term as
\[
\begin{aligned}
\|\bh\|_{\rL_1(B_{3b}\times\BT)}
\leq
C\|\bh\|_{\rL_p(\BT,\rL_q(B_{3b}))}
&\leq
C\big(
<\bF_\perp>_{p, \ell} + \|\bg_\perp\|_{\rL_p(\BT, \rL_q(\Omega))}
+ \|\CR_{4}\bH_\perp\|_{\rL_p(\BT, \rL_q(\Omega))}\big) \\
&\leq C\big(
<\bF_\perp>_{p, \ell} + \|\bg_\perp\|_{\rL_p(\BT, \rL_q(\Omega))}
+ \|\bH_\perp\|_{\rL_p(\BT, \rL_q(\Omega))}\big),
\end{aligned}
\]
For $|x|\geq 4b$ we now conclude
\[
\|\nabla^m\bv_2(x, \cdot)\|_{\rL_{p}(\BT)}
\leq C|x|^{-3-m}\big(
<\bF_\perp>_{p, \ell} + \|\bg_\perp\|_{\rL_p(\BT, \rL_q(\Omega))}
+ <\dv\bF_\perp>_{p, \ell+1}\big)
\]
in virtue of estimate \eqref{est:4.1}.
Since $\bv=\bv_1+\bv_2$ for $|x|\geq 4b$, this completes the proof of Theorem \ref{thm:main.4.1}.
\end{proof}


\subsection{Existence of periodic solutions}
\label{subsec:ext_domain.final}

The linear theory from Theorem \ref{thm:tpStokes_extdom}
is now a direct consequence of 
Theorem \ref{thm:main3} and Theorem \ref{thm:main.4.1} if $\bh=0$.
For the case of non-zero boundary data $\bh$ we proceed similarly 
to the proof of Theorem \ref{thm:5.2}.

\begin{proof}[Proof of Theorem \ref{thm:tpStokes_extdom}]
At first consider the case $\bh=0$.
Let $(\bv_S,\fp_S)=(\bu,\fp)$ be the unique solution to \eqref{eq:statStokes}
with $\bF=\bG_S$ and $\bg=\bg_S$,
which exists due to Theorem \ref{thm:main3},
and let $(\bv_\perp,\fp_\perp)$
be the unique solution to \eqref{eq:poStokes_extdom},
which exists due to Theorem \ref{thm:osc.1} and Theorem \ref{thm:main.4.1}.
Then $\bv=\bv_S+\bv_\perp$ and $\fp=\fp_S+\fp_\perp$ 
defines a solution $(\bv,\fp)$ to \eqref{eq:tpStokes_extdom} with $\bh=0$,
and \eqref{est:tpStokes_extdom} follows from \eqref{est:main3},
\eqref{est:osc.1}, \eqref{eq:4.2} and \eqref{est:4.1}.

To show existence for arbitrary 
$\bh\in\rH^1_{p}(\BT, \rL_q(\Omega)^N) \cap \rL_{p}(\BT, \rH^2_q(\Omega)^N)$,
we fix $\lambda_1>\lambda_0$ with $\lambda_0$ from Theorem \ref{thm:res.prob.inhom}
and define
\[
\begin{aligned}
\bv_1 
&= \sF^{-1}_\BT[\CS(ik+\lambda_1)\big(0, (ik+\lambda_1)\tilde\bh_k, 
(ik+\lambda_1)^{1/2}\tilde\bh_k, \tilde\bh_k\big)], 
\\
\fp_1 
&= \sF^{-1}_\BT[\CP(ik+\lambda_1)\big(0, (ik+\lambda_1)\tilde\bh_k, 
(ik+\lambda_1)^{1/2}\tilde\bh_k, \tilde\bh_k\big)],
\end{aligned}
\]
where $\CS$ and $\CP$ are the $\sR$-bounded solution operators from Theorem \ref{thm:res.prob.inhom},
and $\tilde\bh_k=\sF[\bh](k)$.
Then $(\bv_1,\fp_1)$ is a solution to
the auxiliary problem 
$$\pd_t\bv_1+\lambda_1\bv_1 - \mu\Delta\bv_1 + \nabla\fp_1 
= 0,
\quad \dv\bv_1 = 0\quad
\text{in $\Omega\times\BT$}, \quad 
\bv_1|_\Gamma  = \bh|_\Gamma.
$$
Following the proof of Theorem \ref{thm:5.2} and invoking
Corollary \ref{cor:transferenceprinciple}, 
we further conclude
\[
\bv_1 \in \rH^1_p(\BT, \rL_q(\Omega)^N) \cap \rL_p(\BT, \rH^2_q(\Omega)^N), 
\quad
\fp_1 \in \rL_p(\BT, \hat \rH^1_q(\Omega))
\]
and the estimate
\[
\|\pd_t\bv_1\|_{\rL_p(\BT, \rL_q(\Omega))}
+ \|\bv_1\|_{\rL_p(\BT, \rH^2_q(\Omega))}
+ \|\nabla \fp_1\|_{\rL_p(\BT, \rL_q(\Omega))}
\leq C(\|\pd_t\bh\|_{\rL_p(\BT, \rL_q(\Omega))} + \|\bh\|_{\rL_p(\BT, \rH^2_q(\Omega))}).
\]
Now let $\varphi\in C^\infty_0(\Omega)$ 
such that $\varphi\equiv 1$ in $B_{2b}$ and $\varphi\equiv 0$ in $\BR^3\setminus B_{3b}$. 
Let $D_{2b, 3b} = \{x \in \BR^N \mid 2b < |x| < 3b\}$ and 
$$
\rH^2_{q, 0, a}(D_{2b, 3b}) =\{f \in \rH^2_q(D_{2b, 3b}) \mid \pd_x^\alpha f|_{S_L}=0
\enskip \text{for $L=2b$, $3b$ and $|\alpha|\leq 1$}, \enskip 
\int_{D_{2b, 3b}} f(x)\,\dd x = 0\}.
$$
According to \cite[Lemma 5]{Shibata18}, we know that 
$(\nabla\varphi)\cdot\bv_1(t) \in \rH^2_{q, 0, a}(D_{2b, 3b})$ for a.a.~$t\in\BR$,
and setting
$\bw_1 = \varphi \bv_1 - \BB[(\nabla\varphi)\cdot\bv_1]$, we see that
\begin{equation}\label{prop:3}
\begin{aligned}
&\bw_1 \in \rH^1_p(\BT, \rL_q(\Omega)^3) \cap \rL_p(\BT, \rH^2_q(\Omega)^3), 
\quad {\rm supp}\,\bw_1 \subset B_{3b} \cap \Omega, 
\quad \dv\bw_1=0,
\quad \bw_1|_\Gamma = \bh,\\
&\|\pd_t\bw_1\|_{\rL_p(\BT, \rL_q(\Omega))}
+ \|\bw_1\|_{\rL_p(\BT, \rH^2_q(\Omega))}
\leq C(\|\pd_t\bh\|_{\rL_p(\BT, \rL_q(\Omega))} + \|\bh\|_{\rL_p(\BT, \rH^2_q(\Omega))}).
\end{aligned}
\end{equation}
Now let $(\bw_2,\fq_2)$ be the unique solution to
\[
\pd_t\bw_2 - \Delta\bw_2 + \nabla \fq_2 =
\bff-\pd_t\bw_1 - \Delta\bw_1,
\quad \dv \bw_2 = 0
\quad\text{in
$\Omega\times\BT$}, \quad
\bw_2|_\Gamma=0,
\]
which exists due to the first part of the proof. 
Note that $\bw_1$ vanishes in $\BR^3\setminus B_{3b}$,
so that $(\bv,\fp)=(\bw_1+\bw_2,\fq_2)$
is a solution to \eqref{eq:tpStokes_extdom},
and estimate \eqref{est:tpStokes_extdom} follows from
the corresponding estimate for $\bw_2$ and the properties listed in \eqref{prop:3}.

The uniqueness assertion follows by decomposing a solution $(\bv,\fp)$
into a stationary and an oscillatory part by means of \eqref{eq:decomposition}
and employing the uniqueness statements from 
Theorem \ref{thm:main3} and Theorem \ref{thm:osc.1}.
\end{proof}

\begin{proof}[Proof of Theorem \ref{mainthm:extdom}]
We employ Banach's contraction mapping principle. Define
\begin{align*}
\CI_\varepsilon = \{(\bv, \fq) &\mid \bv = \bv_\perp + \bv_S, 
\enskip \fq = 
\fq_\perp + \fq_S, \enskip \bv_\perp
\in \rH^1_p(\BT, \rL_q(\Omega)^3) \cap 
\rL_p(\BT, \rH^2_q(\Omega)^3), \\
&\quad
\bv_S \in \rH^2_q(\Omega)^3, \quad \dv\bv = 0, \quad  
 \fq_\perp \in \rL_p(\BT, \hat\rH^1_q(\Omega)), \enskip
\fq_S \in \rH^1_q(\Omega), \enskip 
\|(\bv, \fq)\|_{\CI_\varepsilon} \leq \varepsilon\},
\end{align*}
where we set 
\begin{align*}
\|(\bv, \fq)\|_{\CI_\varepsilon} & = \|\pd_t\bv_\perp\|_{\rL_p(\BT, \rL_q(\Omega))}
+ \|\bv_\perp\|_{\rL_p(\BT, \rH^2_q(\Omega))} + \|\bv_S\|_{\rH^2_q(\Omega)}
+ \|\nabla\fq_\perp\|_{\rL_p(\BT, \rL_q(\Omega))} + \|\fq_S\|_{\rH^1_q(\Omega)}
\\
&\qquad+ <\bv_\perp>_{p,1} + <\nabla\bv_\perp>_{p, 2} + 
<\bv_S>_{1} + <\nabla \bv_S>_{2}.
\end{align*}
For $(\bv, \fq) \in \CI_\varepsilon$, let $(\bu, \fp)$ be the solution of
the linear system of equations
\begin{equation}\label{0.6.5.1}
\pd_t\bu - \mu\Delta\bu + \nabla\fp = \bF + \CN(\bv),
\quad\dv\bu=0
\quad\text{in $\Omega\times\BT$}, \quad \bu|_{\Gamma}  = \bh|_\Gamma,
\end{equation}
where $\CN(\bv)=\bv\cdot\nabla\bv$.
Theorem \ref{thm:tpStokes_extdom} yields that 
\begin{equation}
\label{est:0.6.1}
\begin{aligned}
\|(\bu, \fp)\|_{\CI_\varepsilon}
&\leq 
C(<\bG_S>_3 
+ <\bH_S>_2 
+ <\CN(\bv)_S>_3 + <\tilde\CN(\bv)_S>_2
+ <\bG_\perp>_{p,2}
\\
&\quad+ <\bH_\perp>_{p, 1}
+<\CN(\bv)_\perp>_{p, 2} + <\tilde\CN(\bv)_\perp>_{p, 1}
+\|\bh\|_{\rH^1_{p}(\BT, \rL_q(\Omega))}
+\|\bh\|_{\rL_{p}(\BT, \rH^2_q(\Omega))})
\end{aligned}
\end{equation}
provided that the right-hand side of \eqref{est:0.6.1} is
finite.
Here, we 
write $\tilde\CN(\bv) = \bv\otimes\bv$,
so that $\dv\tilde\CN(\bv)=\CN(\bv)$ since $\dv\bv=0$.
We further have 
\begin{equation}\label{nonlinear:1}\begin{aligned}
\CN(\bv)_S &= \bv_S\cdot\nabla\bv_S + \int_\BT
\bv_\perp\cdot\nabla\bv_\perp\,\dd t\\
\tilde\CN(\bv)_S &= \bv_S\otimes\bv_S + \int_\BT
\bv_\perp\otimes\bv_\perp\,\dd t; \\
\CN(\bv)_\perp& = \bv_S\cdot\nabla\bv_\perp + \bv_\perp\cdot\nabla\bv_S
+ \bv_\perp\cdot\nabla\bv_\perp - \int_\BT
\bv_\perp\cdot\nabla\bv_\perp\,\dd t \\
\tilde\CN(\bv)_\perp &= \bv_S\otimes\bv_\perp + \bv_\perp\otimes\bv_S
+ \bv_\perp\otimes\bv_\perp - \int_\BT
\bv_\perp\otimes\bv_\perp\,\dd t.
\end{aligned}\end{equation}
Notice that $\dv\tilde\CN(\bv)_S = \CN(\bv)_S$ and $\dv\tilde\CN(\bv)_\perp
=\CN(\bv)_\perp$. 
To estimate these nonlinear terms, 
we choose $\sigma>0$ so small that $\sigma + 3/q < 2(1-1/p)$,
which is possible since $2/p + 3/q < 2$ by assumption.
By Sobolev inequality and real interpolation,  
we then have
\begin{equation}\label{fest:1}\begin{aligned}
\|\bv_\perp\|_{\rL_\infty(\BT, \rL_\infty(\Omega))} 
&\leq C\|\bv_\perp\|_{\rL_\infty(\BT, \rW^{\sigma + 3/q}_q(\Omega))} 
\leq C\|\bv_\perp\|_{\rL_\infty(\BT, \rB^{2(1-1/p)}_{q,p}(\Omega))} \\
&\leq C(\|\pd_t\bv_\perp\|_{\rL_p(\BT, \rL_q(\Omega))}
+ \|\bv_\perp\|_{\rL_p(\BT, \rH^2_q(\Omega))}),
\end{aligned}\end{equation}
Using H\"older's inequality with $p > 2$, we further obtain
\begin{equation}\label{0fest:5}\begin{aligned}
<\CN(\bv)_S>_3
&\leq C(<\bv_S>_1<\nabla\bv_S>_2 
+ <\bv_\perp>_{p, 1}<\nabla\bv_\perp>_{p, 2}); 
\\
<\tilde\CN(\bv)_S>_2
&\leq C(<\bv_S>_1^2
+ <\bv_\perp>_{p, 1}^2); 
\\
<\CN(\bv)_\perp>_{p, 2} &\leq C(<\bv_S>_{1}
<\nabla\bv_\perp>_{p, 2} + <\bv_\perp>_{p, 1}
<\nabla\bv_S>_2\\
&\qquad+ \|\bv_\perp\|_{\rL_\infty(\BT, \rL_\infty(\Omega))}<\nabla\bv_\perp>_{p, 2}
+<\bv_\perp>_{p, 1}<\nabla\bv_\perp>_{p, 2}); 
\\
<\tilde\CN(\bv)_\perp>_{p, 1} & \leq C(<\bv_S>_1
<\bv_\perp>_{p, 1} + \|\bv_\perp\|_{\rL_\infty(\BT, \rL_\infty(\Omega))}
<\bv_\perp>_{p, 1} + <\bv_\perp>^2_{p, 1}).
\end{aligned}\end{equation}
Combining \eqref{est:0.6.1} 
with \eqref{0fest:5} and \eqref{fest:1} yields that 
\[
\begin{aligned}
\|(\bu, \fp)\|_{\CI_\varepsilon} 
\leq C
(<\bG_S>_3 &+ <\bH_S>_2 + <\bG_\perp>_{p, 2} 
+ <\bH_\perp>_{p, 1}
\\
&\quad
+\|\bh\|_{\rH^1_{p}(\BT, \rL_q(\Omega))}
+\|\bh\|_{\rL_{p}(\BT, \rH^2_q(\Omega))}
+  \|(\bv, \fq)\|_{\CI_\varepsilon}^2).
\end{aligned}
\]
Recalling the smallness assumption \eqref{6:small.2} and that $(\bv, \fq)\in\CI_\varepsilon$, we have
$
\|(\bu, \fp)\|_{\CI_\varepsilon} \leq C_0\varepsilon^2$ for some constant $C_0$.  Thus, choosing
$\varepsilon > 0$ so small that $C_0\varepsilon \leq 1$, we have
$\|(\bu, \fp)\|_{\CI_\varepsilon} \leq \varepsilon$, which implies that 
$(\bu, \fp) \in \CI_\varepsilon$. Therefore, if we define a map $\Xi$ acting on
$(\bv, \fq) \in \CI_\varepsilon$ by setting $\Xi(\bv, \fq) = (\bu, \fp)$,
then $\Xi$ is a map from $\CI_\varepsilon$ into itself.

In an analogous way, we see that for any 
$(\bv_i, \fq_i) \in \CI_\varepsilon$ ($i=1,2$), 
$$\|\Xi(\bv_1, \fq_1)-\Xi(\bv_2, \fq_2)\|_{\CI_\varepsilon}
\leq C_1\varepsilon\|(\bv_1, \fq_1)- (\bv_2, \fq_2)\|_{\CI_\varepsilon}
$$
for some constant $C_1$.  Thus, choosing $\varepsilon > 0$ smaller if necessary,
we have $C_1\varepsilon < 1$, which shows that $\Xi$ is a contraction map
on $\CI_\varepsilon$.  Therefore, there exists a unique $(\bu, \fp) \in \CI_\varepsilon$
such that $\Xi(\bu, \fp)=(\bu_, \fp)$, which is the required unique solution 
to \eqref{5.5}.  This completes the proof of
Theorem \ref{mainthm:extdom}. 
\end{proof}




%


 \end{document}